\documentclass[a4paper,11pt,sort&compress]{elsarticle}
\usepackage{amscd}
\usepackage{epstopdf}
\usepackage{color}
\usepackage{latexsym}
\usepackage{epsf}
\usepackage{multicol}
\usepackage{ifpdf}
\usepackage{amsthm}
\usepackage{amsmath,amsfonts,amssymb,epsfig,subfigure}
\usepackage{mathrsfs}
\usepackage{stmaryrd}
\usepackage{graphicx,amsmath}
\usepackage{latexsym}
\usepackage{verbatim}
\usepackage{listings}
\usepackage{subeqnarray}
\usepackage{cases}
\usepackage{graphicx,natbib,lineno}
\usepackage{float}
\usepackage{amsmath}
\graphicspath{{figure/}}
\ifpdf
\usepackage[%
  pdftitle={Instructions for use of the document class
    elsart},%
  pdfauthor={},%
  pdfsubject={The preprint document class elsart},%
  pdfkeywords={instructions for use, elsart, document class},%
  pdfstartview=FitH,%
  bookmarks=true,%
  bookmarksopen=true,%
  breaklinks=true,%
  colorlinks=true,%
  linkcolor=blue,anchorcolor=blue,%
  citecolor=blue,filecolor=blue,%
  menucolor=blue,pagecolor=blue,%
  urlcolor=blue]{hyperref}
\else
\usepackage[%
  breaklinks=true,%
  colorlinks=true,%
  linkcolor=blue,anchorcolor=blue,%
  citecolor=blue,filecolor=blue,%
  menucolor=blue,pagecolor=blue,%
  urlcolor=blue]{hyperref}
\fi
\makeatletter
\def\elsartstyle{%
    \def\normalsize{\@setfontsize\normalsize\@xiipt{14.5}}
    \def\small{\@setfontsize\small\@xipt{13.6}}
    \let\footnotesize=\small
    \def\large{\@setfontsize\large\@xivpt{18}}
    \def\Large{\@setfontsize\Large\@xviipt{22}}
    \skip\@mpfootins = 18\p@ \@plus 2\p@
    \normalsize
} \@ifundefined{square}{}{\let\Box\square}
\makeatother

\newtheorem{theorem}{Theorem}[section]
\newtheorem{definition}{Definition}[section]

\newtheorem{lemma}[theorem]{Lemma}
\newtheorem{rem}{Remark}[section]
\newtheorem{expl}[theorem]{Example}

\newtheorem{cor}[theorem]{Corollary}

\makeatletter \@addtoreset{equation}{section}

\allowdisplaybreaks
\newcommand{\E}{\mathbb{E}}
\newcommand{\PP}{\mathbb{P}}
\newcommand{\RR}{\mathbb{R}}

\def\nn{\nonumber}

\def\F{{\cal F}}

\def\la{\label}\def\be{\begin{equation}}  \def\ee{\end{equation}}

\usepackage[text={6.5in,9in},centering]{geometry}
\usepackage{diagbox}
\usepackage{longtable}
\usepackage{caption}

\journal{Journal of \LaTeX\ Templates}

\begin{document}
\begin{frontmatter}
\title{ Wolbachia invasion to wild mosquito population in
stochastic environment}

\author[a]{Yuanping Cui}
\ead{cuiyp058@nenu.edu.cn}
\author[b]{Xiaoyue Li\corref{mycorrespondingauthor}}
\cortext[mycorrespondingauthor]{Corresponding author}
\ead{lixy@tiangong.edu.cn}
\author[c]{Xuerong Mao}
\ead{x.mao@strath.ac.uk}
\author[d]{Hongfu Yang}
\ead{yanghf783@nenu.edu.cn}
\address[a]{School of Mathematics and Statistics,
Northeast Normal University, Changchun, 130024, China. }
\address[b]{School of Mathematical Sciences,
Tiangong  University, {Tianjin},  300387, China.}
\address[c]{Department of Mathematics and Statistics,
University of Strathclyde, Glasgow G1 1XH, U.K.}
\address[d]{School of Mathematics and Statistics,
Guangxi Normal University, Guangxi, 541004, China.}





\begin{abstract}
Releasing sterile Wolbachia-infected mosquitoes to invade wild mosquito population is a method of mosquito control. In this paper, a stochastic mosquito population model with Wolbachia invasion  perturbed
by environmental fluctuation is studied. Firstly, well-posedness, positivity and Markov-Feller property of
solution  for this model are proved.
Then a group of sharp threshold-type conditions is provided to characterize the long-term behavior of the model, which pinpoints the almost necessary and sufficient conditions for persistence and extinction of Wolbachia-infected and uninfected mosquito populations. Especially, our results indicates that even the initial Wolbachia-infection
frequency is low, the Wolbachia invasion into wild mosquito population can be promoted by stochastic environmental fluctuations. Finally, some numerical
experiments  are carried out to support our  theoretical results.
\end{abstract}

\begin{keyword}
 { Mosquito population model; Wolbachia; Stochastic environment; Permanence; Extinction;  Stationary distribution.}
\end{keyword}

\end{frontmatter}


\section{Introduction}\label{s-w}
Mosquito-borne diseases (MBDs), such as dengue, Zika, yellow fever, have caused a serious threat to
human health worldwide \cite{Calisher, Kyle}.
Due to the lack of a vaccine or effective therapeutic drugs to combat
 these  MBDs, an effective way to  prevent the spread of MBDs is eliminating the main vector Aedes mosquito.  The traditional control measure  is spraying insecticides \cite{Hemingway}. {However, it only has a short-term effect due to
the growing mosquito resistance to insecticides \cite{Hemingway,Somwang} and
results in a severe environmental damage \cite{Zhang2020}.   An innovative and sustainable
 mosquito control strategy comes from  Wolbachia bacteria which has several  peculiarities \cite{Ong}. This bacteria is maternally inherited;   Wolbachia-infected mosquitoes could block virus transmission; {uninfected}
females are unable to produce offspring with infected males due to cytoplasmic incompatibility, which inhibits the growth of wild mosquito population, see \cite{Field,JH} for more references. Therefore, the reduction of wild mosquitoes can be achieved by releasing the artificially  bred infected mosquitoes  to invade wild mosquito population \cite{F, Hoff, 2011, Bzheng}. }This emerging method  has been  successfully tested in some countries, such as  China,  Australia, the United States, and so on \cite{Ong,  Iturbe,  Rasic, Zhang2020}.
\par
Mathematical models of mosquito population with Wolbachia invasion and their dynamical behaviors help us to understand better how  infected and uninfected mosquitoes evolve
and interact. In 1959, Caspari and
Watson \cite{Cas}  first proposed a discrete  mathematical model to analyze the effect of
cytoplasmic incompatibility on the dynamics of mosquito population. To get closer to reality, more and more mathematical models on Wolbachia spread in mosquito population have been constructed and studied, including  discrete time models \cite{Shi, Zheng}, continuous time  models \cite{Yu2021, Qu, Zhang2020}, delayed  models \cite{Huang, Bo2014}. Recently, Hu et al. \cite{Yu2021} proposed a  mosquito population model with Wolbachia invasion
\begin{equation}\label{e3}
\begin{cases}
\displaystyle \mathrm{d}I(t)=I(t)\big[b_{I}-\delta_I-d_I(I(t)+U(t))\big]\mathrm{d}t,\\
\displaystyle \mathrm{d}U(t)=U(t)\Big[\displaystyle{\frac{b_{U}U(t)}{ I(t)+U(t) }}-\delta_{U}-d_U(I(t)+U(t))\Big]\mathrm{d}t.
\end{cases}
\end{equation}
{Here} $I(t)$ and $U(t)$
denote the numbers of infected  and uninfected mosquitoes at time $t$, {respectively;} $b_I , \delta_I$ and $d_I$ denote the total numbers of
offspring per unit of time and per infected mosquito, the density-independent decay rate
and density-dependent decay rate of  infected mosquito, respectively; Similarly, $b_U, \delta_U$ and $d_U$ represent the three corresponding parameters of uninfected mosquitoes.
\par As a mater of fact, the evolution of mosquitoes is highly linked to various kinds of environmental conditions \cite{Couret,HM}. Due to  severe
disturbances caused by sudden changes of climate conditions, such as the temperature and rainfall seasonal variation, the parameters of model vary with environment changes instead of keeping constant. Therefore,  a few model of differential equations in which the environmental conditions switch
randomly between multiple regimes are  developed, we refer readers to \cite{Hu2015,  Hu2019, HuJDE,  Jia} for more recent studies on randomly switched mosquito population model.
\par However, there are also some continuous fluctuations in a stochastic environment. Many factors embedded in  ecosystems such as
temperature, diet, density, nutrient availability and  water continuously fluctuate with stochastic environment. As a result, it is inevitable that the parameters of the system undergo random variations over time which gives them a stochastic character to some extent \cite{Otero,P}. For a better understanding of Wolbachia spread in mosquito population, we have to take into account continuous fluctuation of stochastic environment in modeling. In addition, Jansen et al. \cite{Jansen} studied stochastic spread of Wolbachia and  pointed out that stochastic effect may promote the Wolbachia spread for a low  initial  frequency of infection.

 {In fact,} the mosquito population is inevitably disturbed by random environment fluctuations. Recall that
the parameters  $\delta_I$ and $\delta_U$ represent density-independent decay rates of $I$ and $U$, respectively, which are sensitive to  the random environmental factors  including temperature, rainfall and so on. In practice
we usually estimate them by   average values plus  error terms obeying normal distributions due to the central limit theorem,  namely,
$$
-\delta_I\mathrm{d} t\rightarrow -\delta_I\mathrm{d} t+\sigma_{I}\mathrm{d}B_1(t),~~-\delta_U\mathrm{d} t\rightarrow-\delta_U\mathrm{d} t+\sigma_{U}\mathrm{d}B_2(t),
$$
where $B_{i}(t),i=1,2$ are standard Brownian motions and independent, and $\sigma_{I}$ and $\sigma_{U}$ denote the intensity of white noises. Thus, the stochastic mosquito population  model with Wolbachia invasion
 is  described by the   stochastic differential equation (SDE)
\begin{equation}\label{e2}
\begin{cases}
\mathrm{d}I(t)=I(t)\big[b_{I}-\delta_I-d_I(I(t)+U(t))\big]\mathrm{d}t+\sigma_II(t)\mathrm{d}B_1(t),\\
\mathrm{d}U(t)=U(t)\Big[\displaystyle{\frac{b_{U}U(t)}{I(t)+U(t)}}-\delta_{U}-d_U\big(I(t)+U(t)\big)\Big]\mathrm{d}t+\sigma_UU(t)\mathrm{d}B_2(t)
\end{cases}
\end{equation}
with initial value~$I(0)=I_{0}\geq 0 $,  $U(0)=U_{0} \geq 0$. Based on the actual background of this model, we assume that $\sigma_{I}$, $\sigma_{U}$ and all other parameters are non-negative. Obviously,  model (\ref{e2}) degenerates into (\ref{e3}) if  $\sigma_{I}=\sigma_{U}=0$.
\par This paper is devoted to investigating the stochastic mosquito population model \eqref{e2} with Wolbachia invasion, which describes the Wolbachia spread in mosquito population in stochastic environment.
    By stochastic Lyapunov analysis,  the well-posedness and positivity as well as the Markov-Feller property of the solution of model \eqref{e2} are studied.
   A group of sharp threshold-type conditions is obtained
to characterize the  dynamical behaviors of  model \eqref{e2}, including the persistence and extinction  as well as stationary distribution.
 Four numerical examples are provided to illustrate our main results.  Our simulations suggest that  the  dynamical behaviors of stochastic model \eqref{e2} could be completely different from those of deterministic  model \eqref{e3}.
  A surprising phenomena  is revealed that environment noises could promote a successful  Wolbachia  invasion into wild mosquito population even for a low initial infection frequency, which implies that environment noises can not be neglected.
In application, the threshold-type conditions enable us to understand the impact  of environmental noises on Wolbachia spread in wild mosquito population and  provide some inspirations for controlling wild mosquito population by releasing  infected mosquitoes.

Compared with the results of deterministic model in \cite{Yu2021}, the dynamical behaviors of  stochastic
model \eqref{e2} { may be} completely different, and environment noise intensities are  crucial  factors for a successful Wolbachia invasion, see  Examples \ref{exam_yhf5.1}-\ref{exam_yhf5.4}.
Specifically,  although for a high initial infection frequency, our theory reveals that a large noise intensity of infected mosquito population may cause itself extinction. On the other hand,  even for a low initial infection frequency, proper noise intensities could drive a successful Wolbachia invasion. These facts also exactly  verify the theory on  stochastic Wolbachia spread  discussed in \cite{Jansen}. Thus, the fact is revealed that the
continuous fluctuations of stochastic environments can not be ignored.

\par The rest of this paper is arranged as follows. Section \ref{sec2} gives the existence and uniqueness of the global positive solution and the Markov-Feller property. Section \ref{sec3} explores the dynamical behaviors of stochastic mosquito population model \eqref{e2}. Section \ref{sec-c4} provides the almost necessary and sufficient threshold-type conditions to characterize the dynamics of stochastic mosquito population model \eqref{e2}, including  persistence and extinction as well as stationary distribution.  Section \ref{sec4} presents a couple of examples and numerical simulations to illustrate our theoretical results.

\section{Global positive solution}\la{sec2}
Throughout this paper, let $(\Omega, \F, \{\F_t\}_{t\geq 0},\PP)$ be a complete filtered probability space with $\{\F_t\}_{t\geq 0}$ satisfying the usual conditions (that is, it is right continuous and $\F_0$ contains all $\PP$-null sets). {$\mathbb{P}_{x}$ and $\mathbb{E}_{x}$ denote the probability and expectation corresponding to the initial value $x$ which may be a vector in $\RR^{2}_{+}$ or a constant in $\RR_{+}$, respectively}. Let $|\cdot|$ denote  the Euclidean norm in $\RR^2$.
 We denote by $\RR_+ = [0, +\infty)$, $\RR_+^{\circ} = (0, +\infty)$, $\RR^2_+ =\RR_+\times \RR_+$, and $\RR^{2,\circ}_+ = \RR_+^{\circ}\times\RR_+^{\circ}$. For any $a, b\in \mathbb{R}$, define
 $a\vee b =\max\{a,b\}$, and $a\wedge b =\min\{a,b\}$. By $\boldsymbol{\delta}_{x}$  denote the Dirac measure with mass at point $x\in \mathbb{R}^{n}$.  {For a set $\mathbb{D}$, we  denote by $\mathbb{D}^c$ its complement, and let $\mathbb{I}_{\{x\in\mathbb{D}\}}=1$ if $x\in \mathbb{D}$ and $0$ otherwise.} We say that a set $\mathbb{D}$ is invariant for the process $(I(t),U(t))$  if $\mathbb{P}_{(I_0,U_0)}\big((I(t),U(t)) \in \mathbb{D}\big) = 1 $ for any $t> 0$ and $(I_0,U_0)\in \mathbb{D}$. {In addition,  for fixed $p>0$, $C_p> 0$, depending on   $p$, is a generic constant which may change from line to line.}
 %
 %

Since $I(t)~\hbox{and}~ U(t)$ represent the numbers of infected and uninfected mosquitoes, both of them should be nonnegative.  The theorem below gives an affirmative answer.
\begin{theorem}\label{th3.1}
For any initial value $(I_0,U_0)\in{\RR^{2}_{+}},$
  there is a unique global solution $(I(t),U(t))$  to \eqref{e2} such that $\mathbb{P}_{(I_0,U_0)}\big((I(t),U(t))\in \mathbb{R}^{2}_{+},~\forall t\geq 0\big)=1$.  Especially, $\mathbb{P}_{(I_0,U_0)}\big((I(t),U(t))\in
  \mathbb{R}^{2, \circ}_{+},~\forall t\geq 0\big)=1$ for any $(I_{0},U_{0})\in{\RR^{2,\circ}_{+}}.$ Furthermore, the solution  $(I(t),U(t))$ is a  Markov-Feller process. 
\end{theorem}
\begin{proof} One observes that $\mathbb{P}_{(0,U_0)}\big(I(t)=0,\forall t\geq0\big)=1$ for any $U_0\geq  0$. Inserting  $I(t)\equiv 0$ for any $t\geq 0$~almost surely (a.s.)~into the second equation of \eqref{e2} yields that
\begin{align*}
\mathrm{d}U(t)=U(t)\big(b_{U}-\delta_{U}-d_UU(t)\big)\mathrm{d}t+\sigma_UU(t)\mathrm{d}B_2(t)~~\mathrm{a.s}.
\end{align*}
Proceeding the  similar argument to \cite[Theorem 2.1]{Li2011} implies that the above equation has a unique strong solution $U(t)\in \RR_+^{\circ} $ for any $t\geq 0,~U_0>0 $. Thus, $\mathbb{P}_{(0,U_0)}\big(I(t)=0,~U(t)>0,~\forall t\geq 0\big)=1$ for any $U_0>0$. Similarly, we can also derive that $\mathbb{P}_{(I_0,0)}\big(I(t)>0,~U(t)=0,~\forall t\geq 0\big)=1$ for any $I_0> 0$, and $\mathbb{P}_{(0,0)}\big(I(t)=0,~U(t)=0,~\forall t\geq 0\big)=1$. Next we focus on the case $(I_0,U_0)\in{\RR^{2,\circ}_+}$.
{ Since the coefficients are local Lipschitz continuous on $\RR^{2,\circ}_{+}$}, by \cite[Theorem 3.3.15, p.91]{Mao2006}, there is a unique
 local solution $(I(t),U(t))$ $(t\in[0,\tau_e))$ with any given initial value $(I_0,U_0)\in{\RR^{2,\circ}_+}$,  {where $\tau_{e}$ is the exit time from $\RR^{2,\circ}_{+}$, namely, it is a stopping time such that
either $\limsup_{t\to\tau_e} (I(t) \vee U(t)) = \infty$ or
$\liminf_{t\to\tau_e} (I(t) \wedge U(t)) = 0$ whenever $\tau_e <\infty$.}
  Choose  a constant $k_{0}\geq 1$ such that~$I_0\in ( {1}/{k_{0}},k_{0})$,  $U_0\in ( {1}/{k_{0}},k_{0})$. For any  $k\geq  k_{0}$, define a stopping time by
\begin{equation*}
\tau_{k}=\inf\left\{t\in[0,\tau_{e}):I(t)\wedge U(t)\leq\frac{1}{k} ~\mbox{ or }~I(t)\vee U(t)\geq k\right\}.
\end{equation*}
For an empty set $\emptyset$, we use the convention~$\inf\emptyset=\infty$. One observes that $\tau_{k}\leq \tau_{e}$  and $\tau_{k}$ is increasing as~$k\rightarrow\infty$. Let $\tau_{\infty}= \lim_{k \rightarrow \infty}\tau_{k}$, clearly,  $\tau_{\infty}\leq\tau_{e} $ a.s. If we can prove that~$\tau_\infty=\infty$ a.s. then $\tau_e =\infty$ a.s.  This  implies that the solution $(I(t),U(t))$ is not only in $\mathbb{R}^{2,\circ}_{+}$ but also global.
Define
$$
V_1(x,y)=(x+1-\log x)+(y+1-\log y), ~~\forall (x,y)\in {\RR^{2,\circ}_+}.
$$
Using the It\^{o} formula yields that
\begin{align}\label{cy1}
\mathbb{E}V_1(I(\tau_{k}\wedge T),U(\tau_{k}\wedge T)) = V_1(I_0,U_0)+\mathbb{E}\displaystyle\int_{0}^{\tau_{k}\wedge T}\mathcal LV_1(I(t),U(t))\mathrm{d}t,
\end{align}
where
\begin{align*}
\mathcal LV_1(x,y):=&(x-1)\big(b_{I}-\delta_{I}-d_I(x+y)\big)+\frac{\sigma_I^{2}}{2}\nn\
\\&~~~+(y-1)\Big(b_{U}\frac{y}{x+y}-\delta_U-d_U(x+y)\Big)+\frac{\sigma_U^2}{2}.
\end{align*}
Using the inequality $u\leq 2(u+1-\log u)$ for any $u>0$, we infer that
\begin{align*}
\mathcal LV_1(x,y)&\leq (b_I+d_I+d_U)x+(b_U+d_U+d_I)y
+\delta_I+\delta_U+\frac{\sigma_I^2}{2}+\frac{\sigma_U^2}{2}\nn\
\\&\leq 2(b_I+d_I+d_U)(x+1-\log x)+2(b_U+d_U+d_I)(y+1-\log y)\nn\
\\&~~~+\delta_I+\delta_U+\frac{\sigma_I^2}{2}+\frac{\sigma_U^2}{2}\nn\
\\&\leq 2v_1V_1(x,y)+ v_2,
\end{align*}
where $v_{1}=(b_I+d_I+d_U)\vee(b_U+d_U+d_I)$,  $v_{2}=\delta_I
+\delta_U+\sigma_I^2/2+\sigma_U^2/2$.
This together with \eqref{cy1} implies that
\begin{align*}
\mathbb{E}V_1(I(\tau_{k}\wedge T),U(\tau_{k}\wedge T))\leq& V_1(I_0,U_0)+v_{2}{T}+2v_{1}\mathbb{E}\displaystyle\int_{0}^{\tau_{k}\wedge T}V_1(I(t),U(t))\mathrm{d}t\nn\\
\leq& 
V_1(I_0,U_0)+v_{2}{T}+2v_{1}\displaystyle\int_{0}^{T}
\mathbb{E}V_1(I(t\wedge\tau_{k}),U(t\wedge\tau_{k}))\mathrm{d}t.
\end{align*}
Applying the Gronwall inequality yields that
\begin{equation}
\mathbb{E}V_1\big(I(\tau_{k}\wedge T),U(\tau_{k}\wedge T)\big) \leq \big(V_1(I_0,U_0)+v_{2}T\big)e^{2v_{1}T}.\nn\
\end{equation}
Since the remaining proof of $\tau_\infty=\infty$ a.s. is standard, see \cite[Theorem 2.1]{Li2019},  we  omit it. 

Then it remains to prove that the  solution $(I(t),U(t))$ is a Markov-Feller process. {Define $V_2(x,y)=x^2+y^{2}$ for any $(x,y)\in \mathbb{R}^{2}_{+}/\{(0,0)\}$. Then we derive that
\begin{align*}
\mathcal{L}V_2(x,y)\leq \big(2(b_I\vee b_U)+\sigma_I^2\vee \sigma_U^2\big)V_2(x,y),~~\forall (x,y)\in {\mathbb{R}^{2}_{+}/\{(0,0)\}},
\end{align*}
where
\begin{align*}
\mathcal LV_2(x,y):=&2x^2\big(b_{I}-\delta_{I}-d_I(x+y)\big)+\sigma_I^{2}x^2+2y^2\Big(b_{U}\frac{y}{x+y}-\delta_U-d_U(x+y)\Big)+\sigma_U^2y^2.
\end{align*}
Let $\eta_{k}=\inf\{t\geq 0: I(t)\vee U(t)\geq k\}.$ Using the $\mathrm{It\hat{o}}$ formula and the Gronwall inequality derives that for any $T>0$ there exists a constant $C_T$ such that
\begin{align*}
\mathbb{E}V_2(I(T\wedge\eta_{k}),U(T\wedge\eta_{k}))\leq (I_0^2+U_0^2)C_T,
\end{align*}
}
which together with the Markov  inequality implies that for any $H>0$ and $(I_0,U_0)\in [0,H]\times [0,H]$
\begin{align*}
\mathbb{P}\big(\eta_{k}\leq T\big)\leq { \mathbb{P}_{(I_0,U_0)}\Big(V_2(I(T\wedge\eta_k),U(T\wedge\tau_k))\geq 2k^2\Big)\leq \frac{ C_TH^2}{k^2}}\rightarrow0
\end{align*}
as $k\rightarrow\infty$. Then for any $H>0$, $\varepsilon>0$ and $T>0$, we can choose $K=K(H,\varepsilon,T)\geq H$ large enough such that
$  { C_T H^2}/{K^2}<\varepsilon,$
which implies that for any $H>0$, $\varepsilon>0$ and $T>0$
\begin{align*}
\mathbb{P}_{(I_0,U_0)}\big(0\leq I(t)\vee U(t)\leq K,\forall 0\leq t\leq T\big)= \mathbb{P}(\eta_{K}> T)>1-\varepsilon.
\end{align*}
{This together with local Lipschitz continuity of coefficients  on $\RR^{2}_{+}/\{(0,0)\}$ as well as \cite[Theorem 5.1]{Nguyen} implies that solution process $(I(t),U(t))$  is a homogeneous Markov-Feller process. The proof is complete.}
\end{proof}

\section{The main results of
long-time dynamical behavior}\la{sec3}
This section is devoted {to} proving the long-time dynamical behaviors of stochastic mosquito population model \eqref{e2}. {The main idea is to study the limits of Lyapunov exponents of $I(t)$ and $U(t)$ by using the properties of boundary equations and  the weak convergence of the random occupation measure of the solution process.} The validity of this method has been verified sufficiently in  \cite{ning2018, Tuong, Evins2015}.
To proceed, consider   \eqref{e2} on the boundaries $U(t)\equiv 0$ and $I(t)\equiv0$, respectively, described by
\begin{align*}
\mathrm{d} \check{I}(t)= \check{I}(t)(b_{I}-\delta_I-d_I \check{I}(t))\mathrm{d}t+\sigma_I \check{I}(t)\mathrm{d}B_1(t)
\end{align*} with  $ \check{I}(0)=I_0$, and
\begin{align*}
\mathrm{d} \check{U}(t)= \check{U}(t)(b_U-\delta_U-d_U \check{U}(t))\mathrm{d}t+\sigma_U \check{U}(t)\mathrm{d}B_2(t)
\end{align*}
with  $ \check{U}(0)=U_0$. Owing to the  comparison theorem
 \cite[Thoerem 1.1, p.352]{N1989} and  the nonnegativity  of $I(t)$ and $U(t)$, one observes that for any $ t\geq 0$,
\begin{align}\label{cyp3.3}
0\leq I(t)\leq  \check{I}(t),~~~~0\leq U(t)\leq  \check{U}(t)~~~~\mathrm{a.s.}
\end{align}
For  convenience, let
\begin{align}\label{a3.9}
 \lambda_\vartheta:=b_{\vartheta}-\delta_\vartheta-\frac{1}{2}\sigma_\vartheta^2,&~
 ~~~q_\vartheta:=\frac{2\lambda_\vartheta}{\sigma_\vartheta^2},~ ~~~\beta_\vartheta:=\frac{2d_\vartheta}{\sigma^2_\vartheta} ,
\end{align} where the subscript  $\vartheta$ identities $I$ or $U$.
  Using the  techniques similar to {\cite[Lemma 2.1]{Bao2016}},  and the result {\rm ${\hbox{{\cite[Proposition 2.1]{Evins2015}}}}$},  we characterize dynamic behaviors of ${\check{I}}(t)$ and $ {\check{U}}(t)$.
\begin{lemma}\label{l1} The following assertions hold.
 \begin{itemize}
 \item[$(1)$]  For any $p>0$,
$
\lim_{t\rightarrow\infty}\mathbb{E}\big[(\check{I}(t))^{p}\big] \leq C_p$ and~$\lim_{t\rightarrow\infty}\mathbb{E}\big[(\check{U}(t))^{p}\big] \leq C_p$.
 \item[$(2)$]If $\lambda_I<0$ $\big(\lambda_U<0\big)$, then $\lim_{t\rightarrow\infty}\check{I}(t)=0$~$\big(\lim_{t\rightarrow\infty}
     \check{U}(t)=0\big)$~~$\mathrm{a.s.}$
 \item[$(3)$]If $\lambda_I>0$~$\big(\lambda_U>0\big)$, then $\check{I}(t)$ $\big(\check{U}(t)\big)$ has a unique stationary distribution $\mu_I$ $\big(\mu_U\big)$ on $\RR_+^{\circ}$, which is Gamma distribution $Ga(q_I,\beta_I)$ $\big(Ga(q_U,\beta_U)\big)$ with density function $f_{q_I,\beta_I}(x  )$ $\big(f_{q_U,\beta_U}(x)\big)$, where
$ f_{q,\beta}(x)=\beta ^{q }x^{q -1}e^{-\beta  x}/\Gamma(q ),~x>0,$
 and   $\Gamma(\cdot)$ represents the Gamma function. Moreover, the probability distribution  $\mathbb{P}_{I_0}(\check{I}(t)\in \cdot)$~$\big(\mathbb{P}_{U_0}(\check{U}(t)\in \cdot)\big)$  converges weakly to $Ga(q_I,\beta_I)$ $\big(Ga(q_U,\beta_U)\big)$ as $t\rightarrow\infty$.
\end{itemize}
\end{lemma}
\begin{lemma}\label{LC3.2}
If $\lambda_I>0$~$(\lambda_U>0)$, then for any  $p>0$,
\begin{align}\label{l3.16}
\!\lim_{t\rightarrow\infty}\frac{1}{t}\int_{0}^{t}\big(\check{I}(s)\big)^{p}\mathrm{d}s
=\!\int_{\RR_{+}}x^{p}\mu_I(\mathrm{d}x)<\infty~~\Big(\lim_{t\rightarrow\infty}\frac{1}{t}
\int_{0}^{t}\big(\check{U}(s)\big)^{p}
 \mathrm{d}s=\!\int_{\RR_{+}}x^{p}\mu_{U}(\mathrm{d}y)<\infty\Big)~\mathrm{a.s.}
\end{align}
Especially, for $p=1$,
\begin{align}\label{L3.8}
\lim_{t\rightarrow\infty}\frac{1}{t}\int_{0}^{t}\check{I}(s)\mathrm{d}s
=\int_{\RR_{+}}x\mu_I(\mathrm{d}x)=\frac{\lambda_I}{d_I}~~~
\Big(\lim_{t\rightarrow\infty}\frac{1}{t}\int_{0}^{t}\check{U}(s)\mathrm{d}s
=\int_{\RR_{+}}x\mu_{U}(\mathrm{d}y)=\frac{\lambda_U}{d_U}\Big)~~\mathrm{a.s.}
\end{align}
Furthermore,
\begin{align}\label{me3.8}
&\lim_{t\rightarrow\infty}\frac{\ln \check{I}(t)}{t}=0 ~~~
 \Big(\lim_{t\rightarrow\infty}\frac{\ln \check{U}(t)}{t}=0\Big) ~~\mathrm{a.s.}
\end{align}
\end{lemma}
\begin{proof}
For $\lambda_I>0$, using Lemma \ref{l1} implies that $\check{I}(t)$ has a unique stationary distribution $\mu_I$. Then using the strong ergodicity gives that for any $p>0$,
\begin{align*}
\lim_{t\rightarrow\infty}\frac{1}{t}\int_{0}^{t}(\check{I}(s))^{p}\mathrm{d}s&=\int_{\RR_{+}}x^{p}\mu_I(\mathrm{d}x)=\frac{\beta_I^{q_I}}{\Gamma(q_I)}\int_{\RR_{+}}x^{p+q_I-1}e^{-\beta_I x}\mathrm{d}x
 =\frac{\Gamma(p+q_I)}{\beta_I^{p}\Gamma(q_I)}<\infty~~\mathrm{a.s.}
\end{align*}
Especially, for $p=1$, \eqref{L3.8} holds.
The strong law of large numbers \cite[Theorem 1.6, p.16]{Mao2006} implies
 \begin{align}\label{cypp3.6}
 \lim_{t\rightarrow\infty}\frac{B_1(t)}{t}=0~~\mathrm{a.s.}
 \end{align}
This, together with \eqref{L3.8} and the $\mathrm{It\hat{o}}$ formula, derives  that
\begin{align*}
\lim_{t\rightarrow\infty}\frac{\ln \check{I}(t)}{t}=\lambda_I-d_I\lim_{t\rightarrow\infty}
\frac{1}{t}\int_{0}^{t}\check{I}(s)\mathrm{d}s=0 ~~\mathrm{a.s.}
\end{align*}
By proceeding a similar argument, we can infer the desired results on $\check{U}(t)$. The proof is complete.\end{proof}

Intuitively, to determine  whether   $I(t)$ $\big(U(t)\big)$ converges to zero, we consider the  Lyapunov exponent $\frac{\ln I(t)}{t}$ $\big(\frac{\ln U(t)}{t}\big)$.
 Using the $\mathrm{It\hat{o}}$ formula for the first equation of \eqref{e2} gives that
 \begin{align}\label{cypp3.7}
 \frac{\ln I(t)}{t}=\frac{\ln I_0}{t}+\lambda_I
 -d_{I}\frac{1}{t}\int_{0}^{t}\big(I(s)+U(s)\big)\mathrm{d}s+\frac{\sigma_I B_{1}(t)}{t}~~\mathrm{a.s.}
 \end{align}
Similarly,
utilizing the $\mathrm{It\hat{o}}$ formula for the second equation of \eqref{e2}, 
we obtain
\begin{align}\label{eq_yhf1}
 \frac{\ln U(t)}{t}= & \frac{\ln U_0}{t} +\frac{1}{t}
\int_{0}^{t}\left[\frac{b_UU(s)}{I(s)+U(s)}-d_U \big(I(s)+U(s)\big)\right]\mathrm{d}s -\delta_U-\frac{\sigma_U^2}{2}+ \frac{\sigma_U B_2(t)}{t} .
\end{align}

In what follows, we shall estimate the limits of Lyapunov exponents of $I(t)$ and $U(t)$. For this purpose, define the  random occupation measure of solution process $(I(t),U(t))$ by
 \begin{align*}
\mathbf{\Pi}^{t}(\cdot):=\frac{1}{t}\int_{0}^{t}\mathbb{I}_{\{(I(r),U(r))\in \cdot\}}\mathrm{d}r
\end{align*}
for any $t\geq0$. For any Borel set $\mathbb{D}\subset \RR^{2}_{+}$, $\mathbf{\Pi}^{t}(\mathbb{D})$ is the proportion of time that $(I(s),U(s))$ spends in $\mathbb{D}$ for $0\leq s\leq t$.

\begin{lemma}\label{mL3.6}
  $\{\mathbf{\Pi}^{t}(\cdot)\}_{t\geq0}$ is tight a.s., and any weak limit of $\mathbf{\Pi}^{t}(\cdot)$  is an invariant probability measure of solution process $(I(t),U(t))$ a.s.
\end{lemma}
\begin{proof}
If $\lambda_I<0$, by Lemma \ref{l1}, $\lim\limits_{t\rightarrow\infty}\check{I}(t)=0$ a.s., which implies
\begin{align*}
\lim_{t\rightarrow\infty}\frac{1}{t}\int_{0}^{t}\check{I} (s)\mathrm{d}s=0 ~~\mathrm{a.s.}
\end{align*}
This together {with \eqref{cyp3.3} and \eqref{L3.8}} implies that for any $\lambda_I\neq0$,
\begin{align*}
\limsup_{t\rightarrow\infty}\frac{1}{t}\int_{0}^{t} I (s) \mathrm{d}s\leq \lim_{t\rightarrow\infty}\frac{1}{t}\int_{0}^{t}\check{I} (s)\mathrm{d}s= \frac{\lambda_I}{d_I}\vee0~~\mathrm{a.s.}
\end{align*}
Similarly, for any $\lambda_U\neq0$,
\begin{align*}
\limsup_{t\rightarrow\infty}\frac{1}{t}\int_{0}^{t} U(s) \mathrm{d}s\leq\lim_{t\rightarrow\infty}\frac{1}{t}\int_{0}^{t}\check{U} (s)\mathrm{d}s= \frac{\lambda_U}{d_U}\vee 0 ~~\mathrm{a.s.}
\end{align*}
The above inequalities together with the continuity imply that for almost all $\omega$, there exists a positive constant $C:=C(\omega) $ such that for all $t\geq0$,
\begin{align*}
 \frac{1}{t}\int_{0}^{t} \big(I(s)+ U(s)\big)  \mathrm{d}s\leq C .
\end{align*}
Then,  for any $\epsilon>0$ and almost all $\omega\in \Omega$, there exists a positive constant $K:=K(\varepsilon,\omega)$ such that $C/K <\epsilon$. Letting $\mathbb{D}=[0,K]\times [0,K]$, we derive from the above inequality that
 \begin{align*}
 \frac{1}{t}\int_{0}^{t}\mathbb{I}_{\{(I(s),U(s))\in \mathbb{D}^{c}\}}
 \mathrm{d}s\leq \frac{1}{Kt}\int_{0}^{t} \big(I(s)+ U(s)\big) \mathbb{I}_{\{(I(s),U(s))\in \mathbb{D}^{c}\}} \mathrm{d}s \leq \frac{C}{K}<\epsilon,
 \end{align*}
 which implies  that for any $\varepsilon>0$ and $t>0$,
\begin{align*}
 \frac{1}{t}\int_{0}^{t}\mathbb{I}_{\{(I(s),U(s))\in \mathbb{D}\}}
 \mathrm{d}s>1-\epsilon.
 \end{align*}
Therefore, $\{\mathbf{\Pi}^{t}(\cdot)\}_{t\geq0}$ is tight a.s. By \cite[Proposition 9.1]{Evins2015}, any weak-limit of $\{\mathbf{\Pi}^{t}(\cdot)\}_{t\geq 0}$  is an invariant probability measure of  solution process $(I(t),U(t))$ a.s.
\end{proof}
\par
 We begin with proving  the transience of $(I(t),U(t))$ on $\mathbb{R}^{2,\circ}_{+}$, which implies that the process $(I(t),U(t))$ has no invariant probability measure on $\mathbb{R}^{2,\circ}_{+}$. In other words, {the infected and uninfected  mosquito populations don't coexist in stochastic environment in the long term.}
 \begin{lemma}\label{L3.4}
 If $\lambda_I<0$, then $\lim_{t\rightarrow\infty}I(t)=0$ a.s.
 \end{lemma}
 \begin{proof}
 Due to   $(I(t),U(t))\in \RR^{2,\circ}_+$, it follows from \eqref{cypp3.6} and \eqref{cypp3.7} that
 \begin{align*}
 \limsup_{t\rightarrow\infty}\frac{\ln I(t)}{t}\leq \lambda_I~~\mathrm{a.s.}
 \end{align*}
 which implies the desired assertion.
 \end{proof}
 \begin{lemma}\label{L3}
If $\lambda_I>0$, then for any $\varepsilon>0$ and $H>0$, there exists a constant $\gamma^*>0$ such that
 \begin{align*}
\mathbb{P}_{(I_0,U_0)}\Big(\lim_{t\rightarrow \infty}\frac{\ln U(t)}{t}=\lambda\Big)\geq1-\varepsilon
\end{align*}
for any $(I_0,U_0)\in (0,H]\times(0,\gamma^*]$, where $\lambda= -d_U{\lambda_I}/{d_{I}}-\delta_U-{\sigma_U^2}/{2}.$
\end{lemma}

\begin{proof}Since the proof is rather technical,  we divide it into three steps.
\par {\bf \underline{Step 1.}} Prove that $U(t)\rightarrow 0$ as $t\rightarrow 0$ with
 sufficiently large probability.
  Due to $\lambda_I>0$, recalling the definition given by \eqref{a3.9},    one observes that for any fixed $\kappa \in (0,q_{I}/2\wedge 1)$,   
$$ q_I-\kappa-\gamma \beta_I>\frac{q_{I}}{2}-\gamma \beta_I>0, ~~ \forall\gamma \in (0, \lambda_{I}/2d_I).
$$
The continuity of Gamma function $\Gamma(\cdot)$ implies that the function $$\rho_{\kappa}(\gamma):=\frac{\beta_I^\kappa\Gamma(q_I-\kappa-\gamma \beta_I)}{\Gamma(q_I-\gamma \beta_I) } $$ is continuous for $\gamma \in (0, \lambda_{I}/2d_I)$. Let
\begin{align*}
\lambda_{\gamma_0}:= b_U \rho_{\kappa}(\gamma_0) \gamma_0^{\kappa}+d_U\gamma_0
+\lambda<0.
\end{align*}
Then we define an auxiliary process $\hat{I}^{\gamma_0}(t)$ which satisfies
\begin{align}\label{e15}
\mathrm{d}\hat{I}^{\gamma_0}(t)=\hat{I}^{\gamma_0}(t)(b_{I}-\delta_I-d_I\gamma_0-d_I\hat{I}^{\gamma_0}(t))\mathrm{d}t+\sigma_I \hat{I}^{\gamma_0}(t)\mathrm{d}B_1(t)
\end{align}
with  $\hat{I}^{\gamma_0}(0)=I_0>0$. Due to $\lambda_I-d_I\gamma_0>0$, by Lemma \ref{l1} the process $\hat{I}^{\gamma_0}(t)$ has a unique invariant probability measure denoted by {$\mu^{\gamma_0}_{I}(\cdot)$} on $\RR_+^{\circ}$, which is Gamma distribution  $Ga(q_I-\gamma_0 \beta_I,\beta_I)$. Then using the strong ergodicity of $\hat{I}^{\gamma_0}(t)$ yields that
\begin{align}\label{16}
\lim_{t\rightarrow\infty}\frac{1}{t}\int_{0}^{t}
\hat{I}^{\gamma_0}(s)\mathrm{d}s
=\int_{\mathbb{R}_{+}}x\mu_{I}^{\gamma_0}(\mathrm{d}x)
=\frac{\lambda_I}{d_I}-\gamma_0 ~~\mathrm{a.s.}
\end{align}
and
\begin{align}\label{l3.22}
\lim_{t\rightarrow\infty}\frac{1}{t}\int_{0}^{t}
\big(\hat{I}^{\gamma_0}(s)\big)^{-\kappa}\mathrm{d}s=
\int_{\mathbb{R}_{+}}x^{-\kappa}\mu_{I}^{\gamma_0}(\mathrm{d}x)
=\rho_{\kappa}(\gamma_0)~~\mathrm{a.s.}
\end{align}
In view of \eqref{16}, for any $\varepsilon>0$ there exists a subset $\Omega_1\subset\Omega$ and a constant $T_1:=T_1(\varepsilon)>0$  such that $\mathbb{P}(\Omega_1)>1-\varepsilon/4$, where
\begin{align}\label{a12}
\Omega_1=\Big\{\omega\in \Omega:~&\frac{1}{t}\int_{0}^{t}\hat{I}^{\gamma_0}(s)\mathrm{d}s
\geq\frac{\lambda_I}{d_I}-\gamma_0+\frac{\lambda_{\gamma_0}}{4d_U},
~~\forall t\geq T_1\Big\}.
\end{align}
According to \eqref{l3.22}, there exists  a subset $\Omega_2\subset \Omega$ and a constant $T_2:=T_2(\varepsilon)>0$ such that $\mathbb{P}(\Omega_2)>1-\varepsilon/4$, where
\begin{align}\label{a9}
\Omega_2&=\Big\{\omega\in \Omega: 0\leq \frac{1}{t}\int_{0}^{t}\big(\hat{I}^{\gamma_0}(s)\big)^{-\kappa}\mathrm{d}s\leq \rho_{\kappa}(\gamma_0)-\frac{\lambda_{\gamma_0}}{4b_U\gamma_0^{\kappa}},~~\forall t\geq T_2\Big\}.
\end{align}By the strong law of large numbers \cite[Theorem 1.6, p.16]{Mao2006} for $B_2(t)$,
 \begin{align}\label{cypp3.26}
 \lim_{t\rightarrow\infty}\frac{B_2(t)}{t}=0~~\mathrm{a.s.}
 \end{align}
Then there exists a subset $\Omega_3\subset \Omega$ and a constant $T_3:=T_{3}(\varepsilon)>0$ such that $\mathbb{P}(\Omega_3)>1-\varepsilon/4,$ where
\begin{align}\label{a2}
\Omega_3=\Big\{\omega\in \Omega: \frac{|\sigma_U B_2(t)|}{t} \leq -\frac{\lambda_{\gamma_0}}{4}, ~~\forall t\geq T_{3}\Big\}.
\end{align}
Let $\check{T}:=T_1\vee T_2\vee T_3$, and choose $M>b_{U}\check{T}$ sufficiently  large such that $\mathbb{P}(\Omega_4)>1-\varepsilon/4,$ where
\begin{align}\label{e17}
 \Omega_4=\Big\{\omega\in \Omega: |\sigma_UB_2(t)|\leq M-b_U\check{T}, ~~\forall 0\leq t\leq \check{T}\Big\}.
 \end{align}
From the second equation of \eqref{e2},   using the $\mathrm{It\hat{o}}$ formula yields
\begin{align}\label{eq}
U(t)=U_0\exp\left\{\int_{0}^{t}\left(\frac{b_UU(s)}{I(s)+U(s)}
-\delta_U-\frac{\sigma_U^2}{2}-d_U\Big(I(s)+U(s)\Big)\right)
\mathrm{d}s+\sigma_UB_{2}(t)\right\}.
 \end{align}
 Then letting $\gamma^*\in (0,\gamma_0 e^{-M})$,  we derive from  \eqref{e17} and \eqref{eq}  that
 \begin{align}\label{e3.26}
U(t) \leq U_0\exp\left(b_U\check{T}+|\sigma_U B_2(t)|\right)
 \leq U_0e^M<\gamma_0
 \end{align}
for any $t\in [0, \check{T}]$,   $U_0\leq \gamma^*$ and  $\omega\in \Omega_4$.
 Now we define a stopping time
\begin{align*}
\tau=\inf\{t\geq0: |U(t)|\geq \gamma_0\}.
\end{align*}
One observes from  \eqref{e3.26} that $\tau>\check{T}$ for $\omega\in \Omega_4$ and $U_0\leq \gamma^*$. For clarity  we  rewrite the first equation of  \eqref{e2} as
\begin{align}\label{e3.40}
\mathrm{d}I(t)=I(t)\Big(b_I-\delta_I-d_I\gamma_0 -d_II(t)+d_I\big(\gamma_0-U(t)\big)\Big)\mathrm{d}t+\sigma_II(t)\mathrm{d}B_1(t).
\end{align}
Applying a comparison argument  to \eqref{e15} and \eqref{e3.40} implies that
\begin{align}\label{cyp3.36}
 0<\hat{I}^{\gamma_0}(t)\leq I(t)
\end{align}
for $0\leq t< \tau$ a.s.
Moreover, due to the fixed $\kappa \in (0,q_{I}/2\wedge 1)$,  for $0\leq t<\tau$, we know
\begin{align}\label{ae2}
\frac{U(t)}{I(t)+U(t)}\leq \left(\frac{U(t)}{I(t)+U(t)}\right)^{\kappa} \leq \left(\frac{U(t)}{\hat{I}^{\gamma_0}(t)}\right)^{\kappa} \leq\Big(\frac{\gamma_0}{\hat{I}^{\gamma_0}(t)}\Big)^{\kappa}.
\end{align}
Thus  utilizing \eqref{cyp3.36} and \eqref{ae2}, we deduce from \eqref{eq}  that for $0\leq t< \tau$,
\begin{align*}
U(t)\leq  U_0\exp\left\{b_U\int_{0}^{t}\bigg(\frac{\gamma_0}{\hat{I}^{\gamma_0}(s)}\bigg)^{\kappa}\mathrm{d}s
-\Big(\delta_U+\frac{\sigma_U^2}{2}\Big)t-d_U\int_{0}^{t}\hat{I}^{\gamma_0}(s)\mathrm{d}s
+\sigma_UB_2(t)\right\}.
\end{align*}
Combining \eqref{a12}, \eqref{a9} and \eqref{a2}, for  {$\omega\in \hat{\Omega}:=\cap_{j=1}^4\Omega_j$} and $0< U_0\leq \gamma^*$, we have $\check{T}< \tau$ and
\begin{align}\label{a6}
U(t)&\leq U_0\exp\left\{\Big(b_U\rho_{\kappa}(\gamma_0)\gamma_0^{\kappa}
-\frac{\lambda_{\gamma_0}}{4}\Big)t
-\Big(\delta_U+\frac{\sigma_U^2}{2}\Big)t  -d_U\Big(\frac{\lambda_I}{d_I}-\gamma_0+\frac{\lambda_{\gamma_0}}{4d_U}\Big)t
-\frac{\lambda_{\gamma_0}}{4}t\right\}\nn\
\\
&= U_0\exp\left\{\Big(b_U\rho_{\kappa}(\gamma_0)\gamma_0^{\kappa}
+d_U\gamma_0
+\lambda\Big)t
-\frac{3\lambda_{\gamma_0}}{4}t\right\}\nn\\
&= U_0 e^{\lambda_{\gamma_0}t/4}\leq\gamma^* e^{\lambda_{\gamma_0}t/4}<\gamma_0, ~~\forall t\in [\check{T},\tau).
\end{align}
As a result of \eqref{a6}, for $0<U_0\leq \gamma^*$, we must have $\tau=\infty$ for almost all $\omega\in \hat{\Omega}$. We obtain this assertion by a contradiction argument as follows.  Suppose that $\tau=\infty $ for almost all $\omega\in \hat{\Omega}$ doesn't hold. Then there exists a set $\Omega_5\subset\hat{\Omega}$ with $\mathbb{P}(\Omega_5)>0$ such that for $\omega\in \Omega_5$, $\tau<\infty$. Note that we have already proved that $\tau>\check{T}$ for $\omega\in \hat{\Omega}$. In view of \eqref{a6}, $U(t)\leq \gamma^*<\gamma_0$ for any $t\in [\check{T},\tau)$. Since $U(t)$ is continuous a.s.,  for almost all $\omega\in \Omega_5$, we have $U(\tau)=\lim\limits_{t\rightarrow\tau}U(t)\leq \gamma^*<\gamma_0,$ which is a contradiction with the definition of $\tau$. Hence, $\tau=\infty$ for almost all $\omega\in \hat{\Omega}.$ This fact together with \eqref{a6} implies that
\begin{align}
U(t)\leq \gamma^* e^{\lambda_{\gamma_0}t/4},~\mathrm{for~any}~t\geq \check{T},~\omega\in  \hat{\Omega} ~\mathrm{and}~U_0\leq \gamma^*.\nn\
\end{align}
Then for any initial value $(I_0,U_0)\in [0,H]\times(0,\gamma^*]$,
\begin{align}\label{a4}
\lim_{t\rightarrow\infty}U(t)=0~ \mathrm{for ~almost~all}~ \omega\in \hat{\Omega}.
\end{align}
{\bf\underline{ Step 2}.} Prove that  the random occupation measure $\mathbf{\Pi}^{t}(\cdot)$ converges weakly to $\mu_{I}\times \boldsymbol{\delta}_0$ with sufficiently large probability as $t\rightarrow\infty$.
Due to $\lambda_I>0$, in view of Lemma \ref{l1}, the process $\check{I}(t)$ has a unique stationary distribution $\mu_{I}$.
 Thanks to \eqref{a4} and Lemma \ref{mL3.6}, for almost all $\omega\in \hat{\Omega}$, there exists a random probability measure $\pi_1$ on $\RR_+\times\{0\}$ that is a weak-limit of $\mathbf{\Pi}^{t}(\cdot)$, and $\pi_1$ is an invariant probability measure of  solution process $(I(t),U(t))$ for almost all $\omega\in \hat{\Omega}$.  One notices that $\RR_+^{\circ}\times \{0\}$ and $\{(0,0)\}$ are two invariant sets of $(I(t),U(t))$. Furthermore, $\mu_I\times\boldsymbol{\delta}_{0}$ and $\boldsymbol{\delta}_{(0,0)}$ are the unique invariant probability measures of $(I(t),U(t))$
on invariant sets $\RR_+^{\circ}\times\{0\}$ and $\{(0,0)\}$, respectively.  Referring to \cite[p.121]{Kha}, we derive that $\pi_1$ must be the convex combination of  invariant probability measures  $\mu_I\times\boldsymbol{\delta}_0$ and $\boldsymbol{\delta}_{(0,0)}$, i.e.,
$$\pi_1=\theta(\mu_I\times\boldsymbol{\delta}_0)+(1-\theta)\boldsymbol{\delta}_{(0,0)}$$
for almost all $\omega\in \hat{\Omega}$, where  $0\leq \theta\leq 1$ for almost all $\omega\in \hat{\Omega}$.
Then by  this weak convergence and  the uniform  integrability in  \eqref{l3.16} and \eqref{a4}, applying \cite[Lemma 3.1]{Hening2018} we deduce from \eqref{L3.8} that
\begin{align*}
d_I\lim_{t\rightarrow\infty}\frac{1}{t}\int_{0}^{t}(I(s)+U(s))\mathrm{d}s
&=d_I\lim_{t\rightarrow\infty}\int_{\mathbb{R}^{2}_{+}}(x+y)\mathbf{\Pi}^{t}(\mathrm{d}x, \mathrm{d}y)\nn\
\\&=\theta d_I\int_{\mathbb{R}^{2}_+}(x+y)\mu_{I}\times\boldsymbol{\delta}_0
\big(\mathrm{d}x,\mathrm{d}y\big)=\theta\lambda_I~~\hbox{for~almost~all~} \omega\in \hat{\Omega}.
\end{align*}
Then letting $t\rightarrow\infty$ in \eqref{cypp3.7} and then using \eqref{cypp3.6} yield that
\begin{align*}
\lim_{t\rightarrow\infty}\frac{\ln I(t)}{t}&=\lambda_I-d_I\lim_{t\rightarrow\infty}\frac{1}{t}\int_{0}^{t}(I(s)+U(s))\mathrm{d}s=(1-\theta)\lambda_I~~\hbox{for~ almost~all~}~\omega\in \hat{\Omega}.\nn\
\end{align*}
This together with \eqref{cyp3.3} as well as \eqref{me3.8}  implies
\begin{align*}
(1-\theta)\lambda_I\leq 0~~\hbox{for~almost~all~} \omega\in \hat{\Omega}.
\end{align*}
Thus $\theta=1$ for almost all $\omega\in \hat{\Omega}$, which indicates that the random occupation measure $\mathbf{\Pi}^{t}(\cdot)$ converges weakly to $\mu_I\times\boldsymbol{\delta}_0$ as $t\rightarrow\infty$ for almost all $\omega\in \hat{\Omega}$.
\par {\bf \underline{Step 3}.} Prove that $U(t)\rightarrow 0$  at the exponential rate $\lambda$ with
 sufficiently large probability. By the weak convergence of $\mathbf{\Pi}^{t}(\cdot)$ and the uniform integrability in \eqref{l3.16} and \eqref{a4}, using \cite[Lemma 3.1]{Hening2018} and \eqref{L3.8} shows that
\begin{align*}
&\lim_{t\rightarrow\infty}\frac{1}{t}\int_{0}^{t}\Big[\frac{b_{U}U(s)}{I(s)+U(s)}-d_U\big(I(s)+U(s)\big)\Big]\mathrm{d}s\nn\
\\=&\lim_{t\rightarrow\infty}\int_{\mathbb{R}^{2}_{+}}\Big(\frac{b_{U}y}{x+y}-d_U(x+y)\Big)\mathbf{\Pi}^{t}(\mathrm{d}x,\mathrm{d}y)\nn\
\\=&\int_{\mathbb{R}^{2}_{+}}\Big(\frac{b_{U}y}{x+y}-d_U (x+y)\Big) \mu_I\times\boldsymbol{\delta}_0\big(\mathrm{d}x,\mathrm{d}y\big)
\\=&-d_U\frac{\lambda_I}{d_{I}}~~\hbox{for~almost~all~} \omega\in \hat{\Omega}.
\end{align*}
Then letting $t\rightarrow\infty$ in \eqref{eq_yhf1} and utilizing \eqref{cypp3.26} give that
\begin{align*}
\lim_{t\rightarrow\infty}\frac{\ln U(t)}{t}&=\lim_{t\rightarrow\infty}\frac{1}{t}\int_{0}^{t}\Big[\frac{b_{U}U(s)}{I(s)+U(s)}
-d_U\big(I(s)+U(s)\big)\Big]\mathrm{d}s-\delta_U-\frac{\sigma_U^2 }{2}=\lambda
\end{align*}
for almost $\omega\in \hat{\Omega}$ and $(I_0,U_0)\in(0,H]\times (0, \gamma^*]$, which implies the desired result.
\end{proof}
\par Lemmas \ref{L3.4} and \ref{L3} reveal the transience of $(I(t),U(t))$ on invariant set $\mathbb{R}^{2,\circ}_{+}$, which implies that there is no invariant probability measure on $\mathbb{R}^{2,\circ}_{+}$. Next we point out that any invariant probability measure of $(I(t),U(t))$ (if it exists)   assigns all of
its mass to the boundary $\partial \mathbb{R}^{2}_{+}$.
\begin{theorem}\label{L3.6}
For 
$\lambda_I\neq 0$,
any invariant probability measure of process $(I(t),U(t))$ is of the form
\begin{align}\label{l3.30}
 l_1(\mu_I\times \boldsymbol{\delta}_0)+l_2(\boldsymbol{\delta}_0\times \mu_U)+l_3\boldsymbol{\delta}_{(0,0)},
\end{align}
where $ l_1, l_2
$ and $l_3$ are non-negative constants  such that $l_1+l_2+l_3=1.$
Moreover, for  any $(I_0,U_0)\in \mathbb{R}^{2,\circ}_{+}$ and  $p>0$,
\begin{align}\label{e3.13}
\lim_{t\rightarrow\infty}\mathbb{E}\big[\big(I(t)\wedge U(t)\big)^{p}\big]=0.
\end{align}
\end{theorem}
\begin{proof}
 Thanks to \eqref{cyp3.3}, we derive from  Lemma \ref{l1} that for any $p>0$
 \begin{align*}
\lim_{t\rightarrow\infty}\mathbb{E}\Big[\big(I(t)\big)^{p}+\big(U(t)\big)^{p}\Big]&\leq \lim_{t\rightarrow\infty}\mathbb{E}\Big[\big(\check{I}(t)\big)^{p}+\big(\check{U}(t)\big)^{p}\Big]
\leq C_p.
 \end{align*}
 This together with the continuity of $\mathbb{E}\big[(I(t))^{p}\big]$ and $\mathbb{E}\big[(U(t))^{p}\big]$ leads to that
 \begin{align}\label{cor3.34}
 \sup_{t\geq0}\mathbb{E}\Big[\big(I(t)\big)^{p}+\big(U(t)\big)^{p}\Big]<\infty,
 \end{align}
 which implies that the probability distribution $\big\{\mathbb{P}_{(I_0,U_0)}\big((I(t),U(t))\in\cdot\big)\big\}_{t\geq 0}$ is tight in $\mathbb{R}^{2}_{+}$. Using  Theorem \ref{th3.1} and the   Krylov-Bogoliubov theorem \cite[Theorem 7.1, P.94]{Da} yields the existence of invariant probability  measure of $(I(t),U(t))$ on $\mathbb{R}^{2}_{+}$.  By virtue of Lemmas \ref{L3.4} and \ref{L3}, the solution process  $(I(t),U(t))$ is transient on invariant set $\mathbb{R}^{2,\circ}_{+}$. As a result, any invariant probability measure of $(I(t),U(t))$   concentrates
on $\partial \mathbb{R}^{2}_{+}$. One notices that $\mu_I\times\boldsymbol{\delta}_0$, $\boldsymbol{\delta}_0\times \mu_U$ and $\boldsymbol{\delta}_{(0,0)}$ are unique invariant probability measures (if  exist) on invariant sets $\RR_{+}^{\circ}\times \{0\}, \{0\}\times \RR_{+}^{\circ}~\hbox{and}~\{(0,0)\}$, respectively. Referring to \cite[p.121]{Kha}, we derive that any invariant probability measure of   $(I(t),U(t))$ has the form   \eqref{l3.30}.
Assume that \eqref{e3.13} doesn't hold. In fact,   there exists $(I_0,U_0)\in \mathbb{R}^{2,\circ}_{+}$,   $p*>0$, $\varepsilon_0>0$ and a sequence $\{t_{k}\}_{k=1}^{\infty}$ satisfying $\lim\limits_{k\rightarrow\infty}t_k=\infty$ such that
\begin{align}\label{e3.28}
\limsup_{k\rightarrow\infty}
\mathbb{E}_{(I_0,U_0)}\big(I(t_k)\wedge U(t_k)\big)^{p*}\geq\varepsilon_0.
\end{align}
 By the Prokhorov theorem \cite[Theorem 16.3]{MR1876169}, there exists a subsequence  still denoted by $\{t_k\}_{k\geq 0}$ with a notation abuse slightly  such that $\mathbb{P}_{(I_0,U_0)}\big((I(t_k),U(t_k))\in\cdot\big)$ converges weakly to an invariant probability measure  denoted by $\pi_2$ with the form of \eqref{l3.30}. Hence, by virtue of  the uniform integrability in \eqref{cor3.34},  it follows from \cite[Lemma 3.1]{Hening2018} that $$ \lim_{k\rightarrow\infty}
\mathbb{E}_{(I_0,U_0)}\big(I(t_k)\wedge U(t_k)\big)^{p*}=\int_{\mathbb{R}^{2}_{+}}\big(x\wedge y\big)^{p*}\pi_2(\mathrm{d}x,\mathrm{d}y)=0,$$
which contradicts with \eqref{e3.28}. The proof is complete.
\end{proof}
\begin{rem}
Theorem \ref{L3.4} reveals that  in  stochastic environment it is  impossible for infected  and uninfected mosquito populations  to coexist {in the long term.}
\end{rem}
{Then it is natural to ask ``which mosquito population will persist or go extinct". For clarity we give  the definitions of persistence and extinction.
  Referring to the definition of population stochastic persistence in \cite{JDEA,Benaim}, 
  we define  the stochastic persistence for each  mosquito population.
\begin{definition}\label{def1}
The infected  (uninfected)  mosquito population is almost surely stochastically persistent if for any $\varepsilon>0$, there exists a constant $\eta>0$ such that for any initial value $(I_0,U_0)\in \RR^{2,\circ}_{+}$,
\begin{align*}
\limsup_{t\rightarrow\infty}\mathbf{\Pi}^{t}\big((0,\eta)\times \RR_{+}\big)<\varepsilon ~~\Big(\limsup_{t\rightarrow\infty}\mathbf{\Pi}^{t}\big(\RR_{+}\times(0,\eta)\big)<\varepsilon\Big)~~\mathrm{a.s.}
\end{align*}
\end{definition}
This  persistence definition implies that the fraction of time that $I(t)$~$(U(t))$ spends  staying
near extinction state zero is very small. Then  referring to the population extinction defined in \cite{ning2018,Hening2018}, we present the definition of extinction  of each mosquito population.
\begin{definition}
The infected  (uninfected)  mosquito population goes extinct if  for any initial value $(I_0,U_0)\in \RR^{2,\circ}_{+}$,
\begin{align*}
\lim_{t\rightarrow\infty}I(t)=0~~\Big(\lim_{t\rightarrow\infty}U(t)=0\Big)~~\mathrm{a.s.}
\end{align*}
Moreover,  if there exists a constant $\chi<0$ such that for any initial value $(I_0,U_0)\in \RR^{2,\circ}_{+}$,
\begin{align*}
\lim_{t\rightarrow\infty}\frac{\ln I(t)}{t}=\chi~~\Big(\lim_{t\rightarrow\infty}\frac{\ln U(t)}{t}=\chi \Big)~~\mathrm{a.s.,}
\end{align*}
we say that infected (uninfected) mosquito population goes extinct exponentially fast.
\end{definition}}
\par
 Next, we give more precise dynamical characterizations for stochastic mosquito populations.
{\begin{lemma}\label{L3.3}
Assume that  $\lambda_I>0$  and   $\lim_{t\rightarrow\infty}U(t)=0$ a.s. for any  $(I_0,U_0)\in \RR^{2,\circ}_{+}$. Then for any $\varepsilon>0$, there are  constants $T>0$ and $H>1$  such that
\begin{align}
&\mathbb{P}_{(I_0,U_0)}\Big(\frac{1}{H}\leq I(t)\leq H\Big)>1-\varepsilon,~~\forall t\geq T,\nn\
\end{align}
where $T$ depends on $\varepsilon$ and   $(I_0,U_0)$.
\end{lemma}}
 \begin{proof}  { Since the proof is rather technical we divide it into three steps.}

 \underline{Step 1.} ~~In order for the lower bound of $I(t)$ we construct the corresponding comparison equation.
 For any $(I_0,U_0)\in \RR^{2,\circ}_{+}$ and  any $\varepsilon>0$,
owing to the fact $\lim_{t\rightarrow\infty}U(t)=0$ a.s.  there { exists a set $\Omega_\varepsilon\subset\Omega$ and a constant $T_1 =T_1(\varepsilon,I_0,U_0)$} such that  $\mathbb{P}(\Omega_\varepsilon)>1-\varepsilon/3$, where
\begin{align}\label{cyL3.34}
\Omega_\varepsilon=\{\omega\in \Omega: U(t)\leq \varepsilon,~~\forall t\geq T_1\}.
\end{align}
Then we define an auxiliary process $\hat{I}^{\varepsilon}(t)$ on $[T_1, \infty)$ by
\begin{align}\label{cyL3.38}
\mathrm{d}\hat{I}^{\varepsilon}(t)=\hat{I}^{\varepsilon}(t)(b_I-\delta_I-d_I\varepsilon-d_I\hat{I}^{\varepsilon}(t))\mathrm{d}t+\sigma_I \hat{I}^{\varepsilon}(t)\mathrm{d}B_1(t)
\end{align}
with $\hat{I}^{\varepsilon}(T_1)=I(T_1)>0$. It follows from \eqref{cyL3.34} and the comparison theorem
 \cite[Thoerem 1.1, p.352]{N1989} that \begin{align}\label{cyL3.36}
\hat{I}^{\varepsilon}(t)\leq I(t),~~\forall \omega\in \Omega_\epsilon,~t\geq T_1.
\end{align}

  \underline{Step 2.} ~~We analyze uniformly  upper boundedness of $\E(\hat{I}^{\varepsilon}(t))^{-\rho}$   by the stochastic Lyapunov analysis. {Let $\hat{I}^{\varepsilon}(T_1)=\hat{x}$ and $V(x)=(1+1/x)^{\rho},~x\geq0$ for any $\rho\in (0,q_I)$.}  Then choosing a positive constant $\alpha>0$ and using the It\^o formula for \eqref{cyL3.38} yield that
\begin{align*}
M_V(t):=e^{\alpha (t-T_1)}V(\hat{I}^{\varepsilon}(t))-V(\hat{x})-\int_{T_1}^{t}\mathcal{L}\big[e^{\alpha (s-T_1)}V(\hat{I}^{\varepsilon}(s))\big]\mathrm{d}s
\end{align*}
is a local martingale, where
\begin{align*}
\mathcal{L}\big(e^{\alpha (t-T_1)}V(x)\big)&:=\rho e^{\alpha (t-T_1)}\big(1+\frac{1}{x}\big)^{\rho-2}\Big[-\frac{1}{x^2}\Big(\lambda_I-\frac{\rho\sigma_I^2}{2}-d_I\varepsilon-\frac{\alpha}{\rho}\Big)\nn\
\\&~~~+\frac{1}{x}\Big(-b_I+\delta_I+\sigma_I^2+d_I\varepsilon+\frac{2\alpha}{\rho}\Big)
+\frac{\alpha}{\rho}+d_I\Big]\\
&=\rho e^{\alpha (t-T_1)}\theta(x),~~\forall x>0.
\end{align*}
Recalling \eqref{a3.9}, thanks to  $\rho\in (0, q_I)$, we know that $\lambda_I-\rho\sigma_I^2/2>0$. Thus we can choose $\varepsilon=\varepsilon(\rho)$ and $\alpha=\alpha(\rho)$ sufficiently small such that $\lambda_I-\rho\sigma_I^2/2-d_I\varepsilon-\alpha/\rho>0$. Then one observes that
\begin{align*}
\lim_{x\rightarrow0^+}\theta(x)=-\infty~~\hbox{and}~~ \lim_{x\rightarrow\infty}\theta(x)=\frac{\alpha}{\rho}+d_I,
\end{align*}
which together with the continuity of $\theta(x)$ for $x\in (0,\infty)$ implies that
\begin{align*}
L(\rho):=\rho\sup_{x>0}\theta(x)<\infty.
\end{align*}
Thus we derive that for any $x>0$,
\begin{align}\label{cyp3.14}
\mathcal{L}\big(e^{\alpha (t-T_1)}V(x)\big)\leq  L(\rho)e^{\alpha (t-T_1)}.
\end{align}
First, for any  $\hat{I}^{\varepsilon}(T_1)=\hat{x}>0$, let $k_0$ be sufficiently large for  $\hat{x}$ staying within the interval $(1/k_0,k_0)$. Then for any $k>k_0$, define the stopping time
{$$\hat{\tau}^{\varepsilon}_k=\inf\left\{t\geq T_1,~\hat{I}^{\varepsilon}(t)\leq \frac{1}{k}\right\}.$$
Note that $\hat{\tau}^{\varepsilon}_{k}$ }is monotonically increasing {as $k\rightarrow\infty$} and its (finite or infinite) limit is denoted by {$\hat{\tau}^{\varepsilon}_{\infty}$}. Similar to Theorem \ref{th3.1}, we can prove that {$\hat{\tau}^{\varepsilon}_{\infty}=\infty$ a.s.} Making use of the local martingale property implies that
{$\mathbb{E}\big[M_{V}(t\wedge\hat{\tau}^{\varepsilon}_{k})\big]=0$}. Thus, for any $t\geq T_1$,
\begin{align}\label{cypp3.16}
\mathbb{E}_{\hat{x}}\Big[e^{\alpha(t\wedge\hat{\tau}^{\varepsilon}_{k}-T_1)}
V\big(\hat{I}^{\varepsilon}(t)\big)\Big]=V(\hat{x})+\mathbb{E}_{\hat{x}}\Big[\int_{T_1}^{t\wedge\hat{\tau}^{\varepsilon}_{k}}\mathcal{L}\Big(e^{\alpha (s-T_1)}V\big(\hat{I}^{\varepsilon}(s)\big)\Big)\mathrm{d}s\Big].
\end{align}
It follows from the definition of $\hat{\tau}^{\varepsilon}_{k}$ that $e^{\alpha (t\wedge\hat{\tau}^{\varepsilon}_{k}-T_1)}\big(1+1/\hat{I}^{\varepsilon}(t\wedge\hat{\tau}^{\varepsilon}_{k})\big)^{\rho}$ is monotonically increasing {as $k\rightarrow\infty$}. Then letting $k\rightarrow\infty$ indicates that
$$e^{\alpha (t\wedge\hat{\tau}^{\varepsilon}_{k}-T_1)}\Big(1+\frac{1}{\hat{I}^{\varepsilon}(t\wedge\hat{\tau}^{\varepsilon}_{k})}\Big)^{\rho}\uparrow e^{\alpha (t-T_1)}\Big(1+\frac{1}{\hat{I}^{\varepsilon}(t)}\Big)^{\rho}~~\mathrm{a.s.}$$
Employing the monotone convergence theorem shows that
\begin{align}\label{cyp3.16}
\lim_{k\rightarrow\infty}\mathbb{E}_{\hat{x}}\Big[e^{\alpha (t\wedge\hat{\tau}^{\varepsilon}_{k}-T_1)}\Big(1+\frac{1}{\hat{I}^{\varepsilon}(t\wedge\hat{\tau}^{\varepsilon}_{k})}\Big)^{\rho}\Big]\uparrow \mathbb{E}_{\hat{x}}\Big[e^{\alpha (t-T_1)}\Big(1+\frac{1}{\hat{I}^{\varepsilon}(t)}\Big)^{\rho}\Big].
\end{align}
On the other hand, by  \eqref{cyp3.14} we deduce that
\begin{align}\label{cyp3.17}
\mathbb{E}_{\hat{x}}\Big[\int_{T_1}^{t\wedge\hat{\tau}^{\varepsilon}_{k}}\mathcal{L}\Big(e^{\alpha (s-T_1)}V\big(\hat{I}^{\varepsilon}(s)\big)\Big)\mathrm{d}s\Big]\leq \int_{T_1}^{t}L(\rho)e^{\alpha (s-T_1)}\mathrm{d}s\leq \frac{1}{\alpha} L(\rho)e^{\alpha (t-T_1)}.
\end{align}
Then letting $k\rightarrow\infty$ in \eqref{cypp3.16} and utilizing \eqref{cyp3.16} and \eqref{cyp3.17} we derive that
\begin{align*}
e^{\alpha (t-T_1)}\mathbb{E}_{\hat{x}}\Big[\Big(1+\frac{1}{\hat{I}^{\varepsilon}(t)}\Big)^{\rho}\Big]\leq \Big(1+\frac{1}{\hat{x}}\Big)^{\rho}+\frac{1}{\alpha} L(\rho)e^{\alpha (t-T_1)},
\end{align*}
which implies that
\begin{align*}
\mathbb{E}_{\hat{x}}\Big[\big(\hat{I}^{\varepsilon}(t)\big)^{-\rho}\Big]\leq \mathbb{E}_{\hat{x}}\Big[\Big(1+\frac{1}{\hat{I}^{\varepsilon}(t)}\Big)^{\rho}\Big]\leq \Big(1+\frac{1}{\hat{x}}\Big)^{\rho}e^{-\alpha (t-T_1)}+\frac{1}{\alpha}L(\rho).
\end{align*}
Therefore, for any $I_0>0$ and $\hat{I}^{\varepsilon}(T_1)=I(T_1)$, making use of the Markov property and { Chapman-Kolmogorov equation} derives that
\begin{align*}
\mathbb{E}\Big[\big(\hat{I}^{\varepsilon}(t)\big)^{-\rho}\Big]
 &= \int_{\RR_{+}}\mathbb{E}_{\hat{x}}\Big[\big(\hat{I}^{\varepsilon}(t)\big)^{-\rho}\Big]
\mathbb{P}_{I_0}\big(I(T_1)\in \mathrm{d}\hat{x}\big)\nn\
\\&\leq e^{-\alpha (t-T_1)}\int_{\RR_{+}}\Big(1+\frac{1}{\hat{x}}\Big)^{\rho}\mathbb{P}_{I_0}\big(I(T_1)\in \mathrm{d}\hat{x}\big)+\frac{1}{\alpha}L(\rho)\nn\
\\&\leq e^{-\alpha (t-T_1)}\mathbb{E}_{I_0}\Big[\Big(1+\frac{1}{I(T_1)}\Big)^{-\rho}\Big]+\frac{1}{\alpha}L(\rho),
\end{align*}
which implies that for any $I_0>0$ and $\hat{I}^{\varepsilon}(T_1)=I(T_1)$,
\begin{align*}
{ \limsup_{t\rightarrow\infty}}\mathbb{E}\Big[\big(\hat{I}^{\varepsilon}(t)\big)^{-\rho}\Big]\leq \frac{1}{\alpha}L(\rho).
\end{align*}
Thus {there exists a constant $T_2:=T_2(I_0,T_1)>T_1$ such that}
\begin{align*}
\mathbb{E}\Big[\big(\hat{I}^{\varepsilon}(t)\big)^{-\rho}\Big]\leq \frac{1}{\alpha}L(\rho)+1=:L(\alpha,\rho),~~\forall t\geq T_2.
\end{align*}

\underline{Step 3.} By the comparison theorem we obtain the lower  and upper bounds of $I(t)$ in probability.
Utilizing the Chebyshev inequality indicates that
\begin{align*}
\mathbb{P}\left(\hat{I}^{\varepsilon}(t)\leq
\Big(\frac{\varepsilon}{3L(\alpha,\rho)}\Big)^{\frac{1}{\rho}}\right)
&=\mathbb{P}\left( {(\hat{I}^{\varepsilon}(t))^{-1}}\geq
\Big(\frac{3L(\alpha,\rho)}{\varepsilon}\Big)^{\frac{1}{\rho}}\right)\leq \frac{\varepsilon\mathbb{E}\big[\big(\hat{I}^{\varepsilon}(t)\big)^{-\rho}\big]}{3L(\alpha,\rho)}\leq \frac{\varepsilon}{3},~~\forall t\geq T_2,
\end{align*}
which together with \eqref{cyL3.36} yields that
\begin{align}\label{cyp3.19}
\mathbb{P}\left(I(t)\leq\Big(\frac{\varepsilon}{3L(\alpha,\rho)}\Big)^{\frac{1}{\rho}}\right)
&=\mathbb{P}\left(\Big\{I(t)\leq\Big(\frac{\varepsilon}{3L(\alpha,\rho)}\Big)^{\frac{1}{\rho}}\Big\}\cap \Omega_\varepsilon \right)+\mathbb{P}\left(\Big\{I(t)\leq\Big(\frac{\varepsilon}{3L(\alpha,\rho)}\Big)^{\frac{1}{\rho}}\Big\}\cap \Omega^{c}_\varepsilon \right)\nn\
\\&\leq \mathbb{P}\left(\hat{I}^{\varepsilon}(t)\leq
\Big(\frac{\varepsilon}{3L(\alpha,\rho)}\Big)^{\frac{1}{\rho}}\right)+\mathbb{P}( \Omega^{c}_\varepsilon)\leq\frac{\varepsilon}{3}+\frac{\varepsilon}{3}\nn\\
&\leq \frac{2\varepsilon}{3},~~\forall t\geq T_2.
\end{align}
On the other hand, by \eqref{cyp3.3} and Lemma \ref{l1}, for any fixed $p>0$ {there exists a constant $T>T_2$} such that
\begin{align*}
\mathbb{E}\big[(I(t))^{p}\big]\leq \mathbb{E}\big[(\check{I}(t))^{p}\big]\leq C_p,~~\forall t\geq T.
\end{align*}
Choosing a constant $H_1>0$ with $H_1^p>3C_p/\varepsilon$, then utilizing the Chebyshev inequality,  we have
\begin{align}\label{cyp3.20}
\mathbb{P}\big(I(t)\geq H_1\big)\leq \frac{\mathbb{E}\big[(I(t))^{p}\big]}{H_1^{p}}\leq \frac{C_{p}}{H_1^{p}}\leq\frac{\varepsilon}{3},~~\forall t\geq T.
\end{align}
Define  $H=\big(3L(\alpha,\rho)/\varepsilon\big)^{1/p}\vee H_1 $. Using \eqref{cyp3.19} and \eqref{cyp3.20} we obtain that
\begin{align*}
\mathbb{P}\Big(\frac{1}{H}< I(t)< H\Big)&\geq1-\mathbb{P}\Big(I(t)\leq \frac{1}{H}\Big)-\mathbb{P}\big(I(t)\geq H\big)\nn\
\\&\geq1-\frac{2\varepsilon}{3}-\frac{\varepsilon}{3}=1-\varepsilon,~~\forall t\geq T.
\end{align*}
The proof is complete.
\end{proof}
 \begin{theorem}\label{LL1}
For any initial value $(I_0,U_0)\in \mathbb{R}^{2,\circ}_{+}$,   the following assertions hold.
\begin{itemize}
 \item[$(1)$] {If $\lambda_I< 0$ $\big(\lambda_U<0\big)$,
 $
 \lim_{t\rightarrow\infty}I(t)=0~\left(\lim_{t\rightarrow\infty}U(t)=0\right)~~\mathrm{a.s.}
 $}
 \item[$(2)$]If $\lambda_U<0<\lambda_I$,
$
 \lim\limits_{t\rightarrow \infty}\displaystyle \frac{\ln U(t)}{t}= -d_U\frac{\lambda_I}{d_{I}}-\delta_U-\frac{\sigma_U^2}{2}~~ \mathrm{a.s.}
 $
 and  the random  occupation measure $\mathbf{\Pi}^{t}(\cdot)$ converges weakly to $\mu_I\times\boldsymbol{\delta}_{0}$ as $t\rightarrow\infty$ a.s. Furthermore, the probability distribution $\mathbb{P}_{(I_0,U_0)}(I(t)\in \cdot)$  converges weakly to $\mu_I$ as $t\rightarrow\infty$.
 \end{itemize}
  \end{theorem}
  \begin{proof}
  (1) One notices that $(I(t),U(t))$ takes values in $ \RR^{2,\circ}_+$ and  $U(t)/(I(t)+U(t))\leq 1$. Then we derive from \eqref{cypp3.6}-\eqref{eq_yhf1} and \eqref{cypp3.26} that
 \begin{align}\label{cyp3.8}
 \limsup_{t\rightarrow\infty}\frac{\ln I(t)}{t}
\leq \lambda_I~~\mathrm{a.s.},~~\limsup_{t\rightarrow\infty}\frac{\ln U(t)}{t}
\leq \lambda_U~~\mathrm{a.s.},
 \end{align}
which implies the desired results.

  (2) Thanks to $\lambda_U<0$,  it follows from \eqref{cyp3.8} directly that
   \begin{align}\label{L3.33}
  \lim\limits_{t\rightarrow\infty}U(t)=0~~\mathrm{a.s.}
  \end{align}
Due to $\lambda_I>0$,  in light of Lemma \ref{mL3.6} and \eqref{L3.33},  any weak limit of  random occupation measure family $\{\mathbf{\Pi}^{t}(\cdot)\}_{t\geq0}$ denoted by $\pi_2$  must have the form
$${\pi_2=\theta(\mu_I\times\boldsymbol{\delta}_0)+(1-\theta)\boldsymbol{\delta}_{(0,0)}~~\mathrm{a.s.},}$$
where  $0\leq \theta\leq 1$~a.s. Using the similar techniques to  Lemma \ref{L3},
 we prove by contradiction that $\theta=1$ a.s. To avoid the duplication we omit the proof details.
Therefore the  random occupation measure  $\mathbf{\Pi}^{t}(\cdot)$ converges weakly to $\mu_I\times\boldsymbol{\delta}_0$ as $t\rightarrow\infty$ a.s.
By this weak  convergence  and  \eqref{l3.16}, \eqref{L3.8} as well as \eqref{L3.33}, we derive from  \cite[Lemma 3.1]{Hening2018} that
\begin{align*}
&\lim_{t\rightarrow\infty}\frac{1}{t}\int_{0}^{t}\Big[\frac{b_{U}U(s)}{I(s)+U(s)}-d_U\big(I(s)+U(s)\big)\Big]\mathrm{d}s\nn\
\\=&\lim_{t\rightarrow\infty}\int_{\mathbb{R}^{2}_{+}}\Big(\frac{b_{U}y}{x+y}-d_U(x+y)\Big)\mathbf{\Pi}^{t}(\mathrm{d}x,\mathrm{d}y)\nn\
\\=&\int_{\mathbb{R}^{2}_{+}}\Big(\frac{b_{U}y}{x+y}-d_U (x+y)\Big) \mu_I\times\boldsymbol{\delta}_0\big(\mathrm{d}x,\mathrm{d}y\big)=-d_U\frac{\lambda_I}{d_{I}}~~\mathrm{a.s.}
\end{align*}
Then  by letting $t\rightarrow\infty$,  it follows from \eqref{eq_yhf1} and \eqref{cypp3.26} that
\begin{align*}
\lim_{t\rightarrow\infty}\frac{\ln U(t)}{t}&=\lim_{t\rightarrow\infty}\frac{1}{t}\int_{0}^{t}\Big[\frac{b_{U}U(s)}{I(s)+U(s)}-d_U\big(I(s)+U(s)\big)\Big]\mathrm{d}s-\delta_U-\frac{\sigma_U^2}{2}\nn\
\\&= -d_U\frac{\lambda_I}{d_{I}}-\delta_U-\frac{\sigma^2_{U}}{2}~~\mathrm{a.s.}
\end{align*}
 which implies the first desired result.

 In what follows, we prove that for any $(I_0,U_0)\in \RR^{2,\circ}_{+}$, the probability distribution {$\mathbb{P}_{(I_0,U_0)}(I(t)\in \cdot)$  converges weakly to $\mu_I$,}  that is, we  need to show that for any  continuous function $f$ on $\RR_{+}$ with $\sup_{x\in \RR_{+}}|f(x)|\leq 1$,
 \begin{align}\label{CL3.52}
 \lim_{t\rightarrow\infty}\Big|\mathbb{E}_{(I_0,U_0)}f(I(t))-\bar{f}_{I}\Big|=0,~~\forall (I_0,U_0)\in \RR^{2,\circ}_{+},
 \end{align}
where {$\bar{f}_{I}:=\int_{\RR_{+}}f(x)\mu_{I}(\mathrm{d}x)$.} In view of Lemma \ref{L3.3},  for any $\varepsilon>0$ and $(I_0,U_0)\in \RR^{2,\circ}_{+}$, there exist  constants $T_1>0$ and $H_1>1$ such that
\begin{align}\label{CYL3.52}
\mathbb{P}_{(I_0,U_0)}\Big(\frac{1}{H_1}< I(t)<H_1\Big)\geq1-\varepsilon,~~\forall t\geq T_1.
\end{align}
 For any $\delta\geq0$, define a set
 \begin{align*}
 \mathbb{D}_{\delta}:=\Big\{(x,y)\in \RR^2_{+}: \frac{1}{H_1}< x<H_1, y\leq\delta \Big\}.
 \end{align*}
According to  Lemma \ref{l1} (3), for any $I_0>0$, the probability distribution $\mathbb{P}_{I_0}\big(\check{I}(t)\in \cdot\big)$ converges weakly to $\mu_I$ as $t\rightarrow\infty$. Thus  there exists a constant $T_2>0$ such that
 \begin{align}\label{cyp3.52}
 \big|\mathbb{E}_{(I_0,0)}f(I(T_2))-\bar{f}_{I}\big|=\big|\mathbb{E}_{I_0}f(\check{I}(T_2))-\bar{f}_{I}\big|\leq \varepsilon.
 \end{align}
 Thanks to the Feller property of solution process $(I(t),U(t))$, there exists a  small enough constant $\delta=\delta(\varepsilon)$ such that for any $(I_0,U_0), (I_0',U_0')\in \mathbb{R}^{2,\circ}_{+}$ with $ (I_0-I'_0)^2+(U_0-U'_0)^2 \leq \delta^2$,
 \begin{align}\label{cyp3.53}
 \big|\mathbb{E}_{(I_0,U_0)}f(I(T_2),U(T_2))-\mathbb{E}_{(I'_0,U'_0)}f(I(T_2),U(T_2))\big|\leq \varepsilon.
 \end{align}
 As a result, combining \eqref{cyp3.52} and \eqref{cyp3.53} implies that
 \begin{align}\label{cyp3.55}
 \big|\mathbb{E}_{(I_0,U_0)}f(I(T_2))-\bar{f}_{I}\big|&\leq \big|\mathbb{E}_{(I_0,U_0)}f(I(T_2))-\mathbb{E}_{(I_0,0)}f(I(T_2))\big|
 +\big|\mathbb{E}_{(I_0,0)}f(I(T_2))-\bar{f}_{I}\big|\nn\
 \\&\leq 2\varepsilon,~~\forall (I_0,U_0)\in \mathbb{D}_{\delta}.
 \end{align}
 In addition, it follows from \eqref{L3.33} that for any $(I_0,U_0)\in \RR^{2,\circ}_{+}$, there is a constant $T_3\geq T_1$ such that
 \begin{align*}
 \mathbb{P}_{(I_0,U_0)}\big(U(t)\leq \delta\big)>1- \varepsilon, ~~\forall t\geq T_3.
 \end{align*}
 This together with \eqref{CYL3.52} yields that for any $(I_0,U_0)\in \RR^{2,\circ}_{+}$,
 \begin{align}\label{cyp3.56}
 \mathbb{P}_{(I_0,U_0)}\Big(\big(I(t),U(t)\big)\notin \mathbb{D}_{\delta}\Big)\leq 2\varepsilon, ~~\forall t\geq T_3.
 \end{align}
 Then for any $t\geq T_2+T_3$, by the homogeneous Markov property, \eqref{cyp3.55}~and \eqref{cyp3.56}  we deduce that
\begin{align*}
 \big|\mathbb{E}_{(I_0,U_0)}f(I(t))-\bar{f}_{I}\big|
 &=\mathbb{E}\Big[\mathbb{E}_{(I_0,U_0)}\Big(f(I(t))-\bar{f}_{I}
 \big|\mathcal{F}_{t-T_2}\Big)\Big]\nn\
 \\&=\mathbb{E}\Big[\mathbb{E}_{(I(t-T_2),U(t-T_2))}\Big(f(I(T_2))-\bar{f}_{I}\Big)\Big]\nn\
 \\& \leq \mathbb{E}\Big[\mathbb{E}_{(I(t-T_2),U(t-T_2))}\Big(f(I(T_2))-\bar{f}_{I}\Big)
 \mathbb{I}_{\{(I(t-T_2),U(t-T_2))\in \mathbb{D}_{\delta}\}}\Big]\nn\
 \\&~~~+\mathbb{E}\Big[\mathbb{E}_{(I(t-T_2),U(t-T_2))}\Big(f(I(T_2))-\bar{f}_{I}\Big)
 \mathbb{I}_{\{(I(t-T_2),U(t-T_2))\notin \mathbb{D}_{\delta}\}}\Big]\nn\
 \\&\leq 2\varepsilon+2(2\varepsilon)\leq 6\varepsilon,
 \end{align*}
which implies the desired result.
  \end{proof}
{\begin{cor}\label{cor1}
If $\lambda_U<0<\lambda_I$, then infected mosquito population is almost surely stochastically persistent.
\end{cor}
\begin{proof}
For any $\varepsilon>0$, we can choose a constant $\eta=\eta(\varepsilon)>0$ small enough such that
\begin{align*}
\mu_I\times\boldsymbol{\delta}_0\big((0,\eta)\times \RR_{+}\big)=\mu_{I}\big((0,\eta)\big)=\frac{\beta_I ^{q_I }}{\Gamma(q_I )}\int_{0}^{\eta}x^{q_I -1}e^{-\beta_I  x}\mathrm{d}x<\varepsilon.
\end{align*}
In view of Theorem \ref{LL1} (2), we know that $\mathbf{\Pi}^{t}(\cdot)$ converges weakly to $\mu_I\times\boldsymbol{\delta}_{0}$ as $t\rightarrow\infty$  $\mathrm{a.s.}$ Thus we derive that
\begin{align*}
\lim_{t\rightarrow\infty}\mathbf{\Pi}^{t}((0,\eta)\times \RR_{+})=\mu_{I}\times \boldsymbol{\delta}_0\big((0,\eta)\times \RR_{+}\big)<\varepsilon,
\end{align*}
which implies the desired result.
\end{proof}}
\begin{theorem}\label{LL3.11}
If  $0<\lambda_U$ and $\lambda_I<\lambda_U-b_U$,
 \begin{align*}
 \lim\limits_{t\rightarrow \infty}\frac{\ln I(t)}{t}=\lambda_I-d_I\frac{\lambda_U}{d_U}~~ \mathrm{a.s.}
 \end{align*}
 and the  random  occupation measure $\mathbf{\Pi}^{t}(\cdot)$ converges weakly to $\boldsymbol{\delta}_{0}\times \mu_U$  as $t\rightarrow\infty$ a.s.
 \end{theorem}
To prove  Theorem \ref{LL3.11}, we introduce an auxiliary process $\hat{U}^{\varepsilon}(t)$ and give its property. For any $\varepsilon>0$,  let $\hat{U}^{\varepsilon}(t)$  satisfy that
\begin{align}\label{cyp3.34}
\mathrm{d}\hat{U}^{\varepsilon}(t)=\hat{U}^{\varepsilon}(t)\left(\frac{b_U\hat{U}^{\varepsilon}(t)}{\varepsilon+\hat{U}^{\varepsilon}(t)}-\delta_U-d_{U}\varepsilon-d_U\hat{U}^{\varepsilon}(t)\right)\mathrm{d}t+\sigma_{U}\hat{U}^{\varepsilon}(t)\mathrm{d}B_{2}(t)
\end{align}
with $\hat{U}^{\varepsilon}(0)=U_0>0$. In a similar way as shown in the proof of Theorem \ref{th3.1}, system \eqref{cyp3.34} has a unique global solution $\hat{U}^{\varepsilon}(t)\in \mathbb{R}^{\circ}_{+}$ on $t\geq 0$ and its solution is a Markov-Feller process. Moreover, for any $\varepsilon>0$ applying the comparison theorem
 \cite[Thoerem 1.1, p.352]{N1989} yields that  for  any $t\geq0$,
\begin{align}\label{m3.45}
\hat{U}^{\varepsilon}(t)\leq \check{U}(t)~~\mathrm{a.s.}
\end{align}
 Next we go a further to give the asymptotic  property of $\hat{U}^{\varepsilon}(t)$.
\begin{lemma}\label{L3.5}
For any $\varepsilon\in (0,\delta_U/2b_{U})$ and {$U_0>0$,}
\begin{align*}
\lim_{t\rightarrow\infty}\frac{\ln \hat{U}^{\varepsilon}(t)}{t}=-\delta_U-d_U\varepsilon-\frac{\sigma_U^2}{2}~~\mathrm{a.s.}
\end{align*}
\end{lemma}
\begin{proof}
In view of \eqref{cypp3.26}, for any $\delta>0$, there is a set $\Omega_1\subset\Omega$ and a constant $T_1:=T_1(\delta)>0$  such that $\mathbb{P}(\Omega_1)>1-\delta/2$, where
\begin{align}\label{me3.37}
\Omega_1=\Big\{\omega\in \Omega: B_{2}(t)\leq \frac{ d_U \varepsilon t}{ \sigma_U},~~\forall t\geq T_1\Big\}.
\end{align}
Then we  choose a constant $M>b_UT_1$ sufficiently large such that $\mathbb{P}(\Omega_2)>1-\delta/2$, where
\begin{align}\label{me3.38}
\Omega_2=\Big\{\omega\in \Omega:~|\sigma_U B_2(t)|\leq M-b_UT_1, ~~\forall 0\leq t\leq T_1\Big\}.
\end{align}
Using the It\^o formula for \eqref{cyp3.34} yields 
\begin{align}\label{me3.39}
\hat{U}^{\varepsilon}(t)&=U_0\exp\left\{\int_{0}^{t}\frac{b_U \hat{U}^{\varepsilon}(s)}{\varepsilon+\hat{U}^{\varepsilon}(s)}\mathrm{d}s-\delta_Ut-d_U\varepsilon t-\frac{\sigma_U^2 t}{2}-d_U\int_{0}^{t}\hat{U}^{\varepsilon}(s)\mathrm{d}s+\sigma_UB_2(t)\right\}.
\end{align}
Let $U_0\in (0, \varepsilon^2e^{-M})$. Then  combining \eqref{me3.38} and \eqref{me3.39} implies that
\begin{align*}
\hat{U}^{\varepsilon}(t)\leq U_0\exp\Big(b_Ut+\sigma_U B_2(t)\Big)\leq U_0 e^{M}\leq \varepsilon^2
\end{align*}
 for any $0\leq t\leq T_1$, $U_0\in (0, \varepsilon^2e^{-M})$ and $\omega\in \Omega_2$. Define the stopping time
 $$\hat{\tau}=\inf\big\{t\geq0: \hat{U}^{\varepsilon}(t)>\varepsilon^2\big\}.$$
 Obviously,  for any  $U_0\in (0, \varepsilon^2e^{-M})$ and $\omega\in \Omega_2$, $\hat{\tau}>T_1$. And $\hat{U}^{\varepsilon}(t)\leq \varepsilon^2$ for any $0\leq t\leq \hat{\tau}$. Owing to the increasing of $y/(\varepsilon+y)$ with respect to $y>0$, we have
 \begin{align*}
 \frac{\hat{U}^{\varepsilon}(t)}{\varepsilon+\hat{U}^{\varepsilon}(t)}\leq \frac{\varepsilon^2}{\varepsilon+\varepsilon^2}\leq \varepsilon,~~ 0\leq t\leq \hat{\tau}.
 \end{align*}
This together with \eqref{me3.37} implies that for any  $U_0\in (0, \varepsilon^2e^{-M})$,
 \begin{align*}
\hat{ U}^{\varepsilon}(t)&\leq U_0\exp\left\{\int_{0}^{t}\frac{b_U \hat{U}^{\varepsilon}(s)}{\varepsilon+\hat{U}^{\varepsilon}(s)}\mathrm{d}s-\delta_Ut \right\}\nn\
 \\&\leq U_0e^{(b_U\varepsilon-\delta_U)t}, ~~\forall \omega\in \Omega_1\cap\Omega_2,~t\in [T_1,\hat{\tau}).
 \end{align*}
Due to $\varepsilon\in (0,\delta_U/2b_U)$,  for any  $U_0\in (0, \varepsilon^2e^{-M})$,  we derive that
 \begin{align*}
 \hat{U}^{\varepsilon}(t)\leq U_0e^{-\delta_U t/{2}}\leq \varepsilon^2e^{-M}<\varepsilon^2, ~~\forall \omega\in \Omega_1\cap\Omega_2,~t\in [T_1,\hat{\tau}).
 \end{align*}
Then  by a similar argument to proving  $\tau=\infty$ in the proof of Lemma \ref{L3}, we deduce that $\hat{\tau}=\infty$ for almost all $\Omega_1\cap \Omega_2$. Therefore,  for any $U_0\in (0, \varepsilon^2e^{-M})$,
\begin{align*}
 \hat{U}^{\varepsilon}(t)\leq U_0e^{-\delta_U t/{2}}~~\hbox{for~almost~all}~ \omega\in \Omega_1\cap \Omega_2,~\hbox{and}~t\geq T_1,
\end{align*}
which implies that for any $U_0\in (0, \varepsilon^2e^{-M})$,
\begin{align*}
\mathbb{P}_{U_0}\Big(\lim_{t\rightarrow\infty}U^{\varepsilon}(t)=0\Big)\geq \mathbb{P}\big(\Omega_1\cap\Omega_2\big)\geq1-\delta.
\end{align*}
Thus, the process $\hat{U}^{\varepsilon}(t)$ is transient and has no invariant measure in $\mathbb{R}^{\circ}_{+}$.  This indicates that  $\boldsymbol{\delta}_0$ is the unique invariant measure of $U^{\varepsilon}(t)$.
 Then define the random occupation measure of $\hat{U}^{\varepsilon}(t)$
\begin{align*}
\hat{\Pi}^{t,\varepsilon}(\cdot):=\frac{1}{t}\int_{0}^{t}\mathbb{I}_{\{\hat{U}^{\varepsilon}(s)\in \cdot\}}\mathrm{d}s.
\end{align*}
 Due to $\lambda_U>0$, using \eqref{l3.16} and \eqref{m3.45} we have
\begin{align}\label{ml3.54}
\limsup_{t\rightarrow\infty}\frac{1}{t}\int_{0}^{t}\big(\hat{U}^{\varepsilon}(s)\big)^{p}\mathrm{d}s\leq \lim_{t\rightarrow\infty}\frac{1}{t}\int_{0}^{t}\big(\check{U}(s)\big)^{p}\mathrm{d}s<\infty~~\mathrm{a.s.}
\end{align}
which implies that the random occupation measure family $\{\hat{\Pi}^{t,\varepsilon}(\cdot)\}_{t\geq0}$ is tight a.s.
 By {\cite[Lemma 5.7]{ning2018}}, any weak-limit of $\hat{\Pi}^{t,\varepsilon}(\cdot)$ (if it exists) is an invariant probability measure of  $\hat{U}^{\varepsilon}(t)$. Because $\boldsymbol{\delta}_{0}$ is the unique invariant probability measure of $\hat{U}^{\varepsilon}(t)$,  $\hat{\Pi}^{t,\varepsilon}(\cdot)$  converges weakly to $\boldsymbol{\delta}_{0}$ a.s. Using this weak convergence  and  the uniform integrability in \eqref{ml3.54} as well as \cite[Lemma 3.1]{Hening2018} we deduce that
\begin{align}\label{cypp3.24}
\lim_{t\rightarrow\infty}\frac{1}{t}\int_{0}^{t}\frac{\hat{U}^{\varepsilon}(s)}{\varepsilon+\hat{U}^{\varepsilon}(s)}\mathrm{d}s=\lim_{t\rightarrow\infty}\int_{\mathbb{R}_{+}}\frac{y}{\varepsilon+y}\hat{\Pi}^{t,\varepsilon}(\mathrm{d}y)=\int_{\mathbb{R}_{+}}\frac{y}{\varepsilon+y}\boldsymbol{\delta}_{0}(\mathrm{d}y)=0~~\mathrm{a.s.}
\end{align}
and
\begin{align}\label{cypp3.25}
\lim_{t\rightarrow\infty}\frac{1}{t}\int_{0}^{t}\hat{U}^{\varepsilon}(s)\mathrm{d}s=\lim_{t\rightarrow\infty}\int_{\mathbb{R}_{+}}y\hat{\Pi}^{t,\varepsilon}(\mathrm{d}y)=\int_{\mathbb{R}_{+}}y\boldsymbol{\delta}_{0}(\mathrm{d}y)=0~~\mathrm{a.s.}
\end{align}
On the other hand, using the $\mathrm{It\hat{o}}$ formula yields that
\begin{align*}
\frac{\ln \hat{U}^{\varepsilon}(t)}{t}=\frac{U_0}{t}+\frac{1}{t}\int_{0}^{t}\frac{b_U\hat{U}^{\varepsilon}(s)}{\varepsilon+\hat{U}^{\varepsilon}(s)}\mathrm{d}s-\delta_U-d_{U}\varepsilon-\frac{\sigma_U^2}{2}-d_U\frac{1}{t}\int_{0}^{t}\hat{U}^{\varepsilon}(s)\mathrm{d}s+\frac{\sigma_U B_{2}(t)}{t}~~\mathrm{a.s.}
\end{align*}
Then letting $t\rightarrow\infty$, we derive from \eqref{cypp3.26}, \eqref{cypp3.24} and \eqref{cypp3.25}  that
\begin{align*}
\lim_{t\rightarrow\infty}\frac{\ln \hat{U}^{\varepsilon}(t)}{t}=-\delta_U-d_U\varepsilon-\frac{\sigma_U^2}{2}~~\mathrm{a.s.}
\end{align*}
The proof is complete.
\end{proof}
\\{\bf Proof of Theorem \ref{LL3.11} }
{ Due to $\lambda_I<\lambda_U-b_U=-\delta_U-\sigma_U^2/2$, applying Lemma \ref{L3.4} yields that
  \begin{align}\label{L3.7}
  \lim\limits_{t\rightarrow\infty}I(t)=0~~\mathrm{a.s.}
  \end{align}
Then for any  {$\varepsilon\in \big(0, ((\lambda_U-b_U-\lambda_I)/3d_U)\wedge (\delta_U/2b_U)\big)$,} there exists a subset $\Omega_1\subset\Omega$ and a constant $T_1:=T_1(\varepsilon)$ such that $\mathbb{P}(\Omega_1)>1-\varepsilon/2$, where
\begin{align*}
\Omega_1=\big\{\omega\in \Omega: I(t)<\varepsilon,~\forall t\geq T_1\big\}.
\end{align*}
Then let the  process $\hat{U}^{\varepsilon}(t)$ defined by \eqref{cyp3.34} start at time $T_1$ with $\hat{U}^{\varepsilon}(T_1)=U(T_1)$. Invoking the comparison theorem \cite[Thoerem 1.1, p.352]{N1989} yields that
\begin{align}\label{CL3.45}
\hat{U}^{\varepsilon}(t) \leq U(t),~~\forall \omega\in \Omega_{1},~~\forall t\geq T_1.
\end{align}
According to Lemma \ref{L3.5}, there exists a subset $\Omega_2\subset \Omega$ and a constant $T_2:=T_2(\varepsilon )\geq T_1$ such that  $\mathbb{P}(\Omega_2)>1-\varepsilon/2$, where
\begin{align}\label{me3.52}
\Omega_2=\Big\{\omega\in \Omega: \hat{U}^{\varepsilon}(t)\geq e^{(\lambda_U-b_U-2d_U \varepsilon) t},~~\forall t\geq T_2\Big\}.
\end{align}
Combining \eqref{CL3.45} and \eqref{me3.52} yields that
\begin{align}\label{eqq1}
 U(t)\geq e^{(\lambda_U-b_U-2d_U \varepsilon)t}, ~~\forall \omega\in \Omega_1\cap\Omega_2,~~\forall t\geq T_2.
\end{align}
On the other hand, using \eqref{cyp3.8} implies that there exists a constant $T_3:=T_3(\varepsilon,\omega)\geq  T_2$ such that for any $t\geq T_{3}$,
\begin{align}\label{eqq2}
I(t)\leq e^{(\lambda_I+d_U\varepsilon) t} ~~\mathrm{a.s.}
\end{align}
By virtue of \eqref{eqq1}, \eqref{eqq2} and the fact $\lambda_U-b_U-\lambda_I-3d_U\varepsilon>0$, we derive that for almost all $\omega\in \Omega_1\cap\Omega_2$}
\begin{align}\label{Cl3.63}
\limsup_{t\rightarrow\infty}\frac{1}{t}\int_{0}^{t}\frac{I(s)}{I(s)+U(s)}\mathrm{d}s
 &\leq \limsup_{t\rightarrow\infty}\frac{1}{t}\int_{T_3}^{\infty}\frac{I(s)}{I(s)+U(s)}\mathrm{d}s\nn\
\\&\leq\limsup_{t\rightarrow\infty}\frac{1}{t}\int_{T_3}^{\infty}\frac{I(s)}{I(s)+\hat{U}^{\varepsilon}(s)}\mathrm{d}s\nn\
\\& \leq\limsup_{t\rightarrow\infty}\frac{1}{t}\int_{T_3}^{\infty}\frac{e^{(\lambda_I+d_U \varepsilon) s}}{e^{(\lambda_I+d_{U}\varepsilon) s}+e^{(\lambda_U-b_U-2d_U \varepsilon)s }}\mathrm{d}s\nn\
\\&= \limsup_{t\rightarrow\infty}\frac{1}{t}\int_{T_3}^{\infty}
\frac{1}{1+e^{(\lambda_U-b_U-\lambda_I-3d_U\varepsilon)s }}\mathrm{d}s=0.
\end{align}
On the other hand,  using \eqref{L3.7},  Lemma \ref{mL3.6} and Theorem \ref{L3.6}, we know that any weak limit of the random occupation measure family $\{\mathbf{\Pi}^{t}(\cdot)\}_{t\geq0}$ denoted by $\pi_3$ has the form  $\pi_3=\theta\boldsymbol{\delta}_{(0,0)}+(1-\theta)(\boldsymbol{\delta}_0\times\mu_U)$, $0\leq \theta\leq 1$ a.s.  Due to the weak convergence and the uniform integrability in   \eqref{l3.16} and \eqref{L3.7}, by virtue of  \cite[Lemma 3.1]{Hening2018} we derive that
 \begin{align}\label{m3.56}
 \lim_{t\rightarrow\infty}\frac{1}{t}\int_{0}^{t}(I(s)+U(s))\mathrm{d}s&=\lim_{t\rightarrow\infty}\int_{\mathbb{R}^2_{+}}(x+y)\mathbf{\Pi}^{t}(\mathrm{d}x,\mathrm{d}y)\nn\
 \\ &=\int_{\mathbb{R}^2_{+}}(x+y)\pi_{3}(\mathrm{d}x,\mathrm{d}y)=(1-\theta)\frac{\lambda_U}{d_U}~~\mathrm{a.s.}
 \end{align}
 Then letting $t\rightarrow\infty$ in \eqref{eq_yhf1} and  invoking \eqref{cypp3.26}, \eqref{Cl3.63}  and \eqref{m3.56}, we deduce that for almost all $\omega\in \Omega_1\cap\Omega_2$,
 \begin{align}\label{me3.57}
 \liminf_{t\rightarrow\infty}\frac{\ln U(t)}{t}= &\liminf_{t\rightarrow\infty}\frac{1}{t}
\int_{0}^{t}\frac{b_UU(s)}{I(s)+U(s)}\mathrm{d}s -\delta_U-\frac{\sigma_U^2}{2}-d_U\lim_{t\rightarrow\infty}
\frac{1}{t}\int_{0}^{t}\big(I(s)+U(s)\big)\mathrm{d}s\nn\
\\ \geq& b_{U}-\limsup_{t\rightarrow\infty}\frac{1}{t}
\int_{0}^{t}\frac{b_UI(s)}{I(s)+U(s)}\mathrm{d}s -\delta_U-\frac{\sigma_U^2}{2}-(1-\theta)\lambda_U = \theta \lambda_U.
 \end{align}
 On the other hand, utilizing  \eqref{cyp3.3} and \eqref{me3.8}  implies that
 \begin{align}
\liminf_{t\rightarrow\infty}\frac{\ln U(t)}{t}\leq \lim_{t\rightarrow\infty}\frac{\ln \check{U}(t)}{t}= 0~~\mathrm{a.s.}\nn\
\end{align}
{Combining this  and \eqref{me3.57} derives that $\theta=0$~for almost all $\omega\in \Omega_1\cap\Omega_2$. According to the definitions of $\Omega_1$ and $\Omega_2$, we know that $\mathbb{P}(\Omega_1\cap\Omega_2)>1-\varepsilon$. Since $\varepsilon$ is an arbitrarily small constant, the random occupation measure $\mathbf{\Pi}^{t}(\cdot)$ must converge weakly to $\boldsymbol{\delta}_0\times\mu_U$ as $t\rightarrow\infty$ a.s.} Then it follows from \eqref{cypp3.6}, \eqref{cypp3.7} and \eqref{m3.56} with $\theta=0$ a.s. that
\begin{align*}
\lim_{t\rightarrow\infty}\frac{\ln I(t)}{t}=\lambda_I-d_I\frac{\lambda_U}{d_U}~~\mathrm{a.s.}
\end{align*}
The proof is complete.\qed
  {\begin{cor}\label{cor2}
 If $0<\lambda_U$ and $\lambda_I<\lambda_U-b_U$, then  uninfected mosquito population is almost surely stochastically persistent.
  \end{cor}}
 \par {On the other hand,  one notices from Lemma \ref{L3} that solutions starting nearby $U(0)=0$ are attracted to the boundary $(0,\infty)\times \{0\}$.
Then we go a further step to explore the asymptotic behaviors of  \eqref{e2} under  $\lambda_U>0$ and $\lambda_I>0$.}
\begin{theorem}\label{t1}
If $\lambda_I/d_I>\lambda_U/d_U>0$, then for any initial value  $(I_0,U_0)\in \mathbb{R}^{2,\mathrm{o}}_{+}$
\begin{align*}
\lim_{t\rightarrow\infty}\frac{\ln U(t)}{t}=-d_U\frac{\lambda_I}{d_I} -\delta_U-\frac{\sigma_U^2}{2}~~\mathrm{a.s.}
\end{align*}
and the random occupation measure $\mathbf{\Pi}^{t}(\cdot)$ converges weakly to $\mu_I\times\boldsymbol{\delta}_0$ as $t\rightarrow\infty$ a.s. Furthermore, the probability distribution $\mathbb{P}_{(I_{0},U_0)}(I(t)\in \cdot)$ converges weakly  to $\mu_I$ as $t\rightarrow\infty$.
\end{theorem}
\begin{proof}
 According to Lemma \ref{mL3.6} and Theorem \ref{L3.6}, the random occupation measure family $\{\mathbf{\Pi}^{t}(\cdot)\}_{t\geq0}$  is tight  on $\mathbb{R}_{+}^{2}$ a.s. and its any weak limit denoted by $\pi_4$ is of the form
 {$$ \theta_1(\mu_I\times \boldsymbol{\delta}_0)+\theta_2(\boldsymbol{\delta}_0\times \mu_U)+\theta_3\boldsymbol{\delta}_{(0,0)}~~\mathrm{a.s.},$$
   here $0\leq \theta_i\leq 1$~a.s. $i=1,2,3$.} By this weak convergence, \eqref{l3.16} and \cite[Lemma 3.1]{Hening2018}, we deduce that for any $(I_0, U_0)\in \mathbb{R}^{2,\circ}_{+}$,
\begin{align*}
\lim_{t\rightarrow\infty}\frac{1}{t}\int_{0}^{t}\big(I(s)+U(s)\big)\mathrm{d}s&=\lim_{t\rightarrow\infty} \int_{\mathbb{R}^{2}_{+}}(x+y)\mathbf{\Pi}^{t}(\mathrm{d}x,\mathrm{d}y)\nn
\\&=\theta_1 \int_{\mathbb{R}_{+}}x\mu_{I}(\mathrm{d}x)+\theta_2 \int_{\mathbb{R}_{+}}y\mu_{U}(\mathrm{d}y)\nn\\
&=
\theta_1\frac{\lambda_I}{d_I}+\theta_2\frac{\lambda_U}{d_U} ~~\mathrm{a.s.}
\end{align*}
Using \eqref{cypp3.6}, \eqref{cypp3.7} and the fact $\theta_1+\theta_2+\theta_3=1$~a.s. derives that
\begin{align}\label{cyp3.44}
\lim_{t\rightarrow\infty}\frac{\ln I(t)}{t}&=\lambda_I-d_I
\lim_{t\rightarrow\infty}\frac{1}{t}\int_{0}^{t}\big( I(s)+U(s)\big)\mathrm{d}s\nn
\\&= (1-\theta_1)\lambda_I-\theta_2\frac{d_I\lambda_U}{d_U}\nn
\\&=\theta_2\left(\lambda_I-\frac{d_I\lambda_U}{d_{U}}\right)+\theta_3\lambda_I~~\mathrm{a.s.}
\end{align}
On the other hand,
combining \eqref{cyp3.3}, \eqref{me3.8} and \eqref{cyp3.44} indicates that
\begin{align*}
\theta_{2}\Big(\lambda_I-\frac{d_I\lambda_U}{d_{U}}\Big)+\theta_3\lambda_I\leq \lim_{t\rightarrow\infty}\frac{\ln \check{I}(t)}{t} =0~~\mathrm{a.s.}
\end{align*}
 This, together with the fact $\lambda_I/d_I>\lambda_U/d_U>0$, implies that $\theta_2=\theta_3=0$ a.s. Thus, $\theta_1=1$ a.s. As a result, for any $(I_0, U_0)\in \mathbb{R}^{2,\circ}_{+}$, $\mathbf{\Pi}^{t}(\cdot)$ converges weakly to measure $\mu_I\times\boldsymbol{\delta}_0$  as $t\rightarrow\infty$ a.s. By this weak convergence and uniform integrability in  \eqref{l3.16}, it follows from \eqref{eq_yhf1}, \eqref{cypp3.26} and \cite[Lemma 3.1]{Hening2018} that
 \begin{align*}
\lim_{t\rightarrow\infty}\frac{\ln U(t)}{t}&=\lim_{t\rightarrow\infty}\frac{1}{t}
\int_{0}^{t}\left(\frac{b_UU(s)}{I(s)+U(s)} -d_U\big(I(s)+U(s)\big)\right)\mathrm{d}s-\delta_U-\frac{\sigma_U^2}{2}\nn\
\\&=\int_{\RR^2_{+}}\Big(\frac{b_Uy}{x+y}-d_{U}(x+y)\Big)\mu_{I}\times\boldsymbol{\delta}_0(\mathrm{d}x,\mathrm{d}y)-\delta_U-\frac{\sigma_U^2}{2}\nn\
\\&=-d_U\frac{\lambda_I}{d_I}-\delta_U-\frac{\sigma_U^2}{2}~~\mathrm{a.s.},
 \end{align*}
 which implies that $\lim_{t\rightarrow\infty}U(t)=0$ a.s. Then by virtue of Lemmas \ref{l1} and \ref{L3.3},  repeating the argument of proving \eqref{CL3.52} implies the desired result. The proof is complete.
\end{proof}
{\begin{cor}\label{cor3}
If $\lambda_I/d_I>\lambda_U/d_U>0$, then infected mosquito population is almost surely stochastically persistent.
\end{cor}}
\section{A group of sharp threshold-type conditions}\label{sec-c4}
\par { Collecting serval results presented in separate theorems and corollaries in Section \ref{sec3}, a group of sharp threshold-type conditions  is provided to characterize the dynamical behaviors of stochastic mosquito population model \eqref{e2}.

\begin{theorem}\la{Th_yhf1}
For stochastic mosquito population model \eqref{e2}, let initial value $(I_0,U_0)\in \mathbb{R}^{2,\circ}_{+}$.
\begin{description}
 \item[(A)]
 For $\lambda_U<0$, the following results hold.
 \begin{description}
 \item[(A.1)]If $\lambda_I<0$, then  both infected and uninfected mosquito populations go extinct.
\item[(A.2)]If $\lambda_I>0$,   then  infected mosquito population is almost surely stochastically persistent, and uninfected mosquito population goes  extinct  exponentially fast. Furthermore, the probability distribution $\mathbb{P}_{(I_0,U_0)}\big(I(t)\in \cdot\big)$   converges weakly to $\mu_I$ as $t\rightarrow\infty$.
\end{description}
 \item[(B)] For $\lambda_U>0$, the following results hold.
 \begin{description}
 \item[(B.1)]If ${\lambda_I}/{d_I}>{\lambda_U}/{d_U}$, then infected mosquito population is almost surely stochastically persistent, and uninfected mosquito population goes extinct exponentially fast. Furthermore, the probability distribution $\mathbb{P}_{(I_0,U_0)}\big(I(t)\in \cdot\big)$   converges weakly to $\mu_I$ as $t\rightarrow\infty$.
  \item[(B.2)] If $\lambda_I<\lambda_U-b_U$, then infected mosquito population goes extinct exponentially fast, and uninfected mosquito population is almost surely stochastically persistent.
 \item[(B.3)]If $\lambda_U-b_U\leq\lambda_I<0$ or $0<\lambda_I/d_I\leq \lambda_U/d_U$, then any stationary distribution of $(I(t),U(t))$  has the form of \eqref{l3.30}.
\end{description}
 \end{description}
\end{theorem}}


\section{Numerical examples}\la{sec4}


In this section, we mainly provide serval numerical examples to illustrate the effect of environment noise on long-time dynamical behaviors of infected and uninfected mosquitoes.
 For convenience to compare with deterministic model \eqref{e3}, we select the same parameter values as that in  \cite{Yu2021}, see Table \ref{table_yhf1}.
\begin{table}[h]
\caption{ The significance and  value of parameters. }
  \label{table_yhf1}
\begin{tabular}{cllllll}
\hline
 \bf{\small Parameter}&~\bf{\small Value}~&~ \bf{Biological significance of parameter}  \\
\hline
$b_I$&~~0.45~&~~the total numbers of
offspring per unit of time, per infected mosquito\\
$b_U$&$~~0.55~$&~~the total numbers of
offspring per unit of time, per wild mosquito\\
$\delta_I$&~$~0.05~$&~~the density-independent decay rate of infected mosquito\\
$\delta_U$&~~0.048~&~~the density-independent decay rate of wild mosquito\\
$d_I$&~$~0.001~$ &~~density-dependent decay rate of
  infected mosquito\\
$d_U$&~$~0.001~$ &~~density-dependent decay rate of
  wild mosquito\\
\hline
\end{tabular}
\end{table}

{Applying the truncated Euler-Maruyama method in \cite{Xiao} yields the discrete equation as
follows
\begin{equation}
\begin{cases}
\displaystyle \tilde{I}_{k+1}=I_{k}+I_{k}\Big[b_{I}-\delta_I-d_I(I_k+U_k)\Big]\Delta+\sigma_II_k\sqrt{\Delta}\zeta_k,\nn\
\\ \displaystyle \tilde{U}_{k+1}=U_k+U_{k}\Big[ \frac{b_{U} U_{k}}{I_k+U_{k}}-\delta_U-d_U(I_{k}+U_{k})\Big]\Delta+\sigma_{U}U_{k}\sqrt{\Delta}\xi_{k},
\\I_{k+1}=\Big[1\wedge\big(600+I_0+U_0\big)
\Delta^{-\frac{2}{5}}\big(\tilde{I}_{k+1}^2+\tilde{U}_{k+1}^2\big)^{-1/2}\Big]
\tilde{I}_{k+1},
\\U_{k+1}=\Big[1\wedge\big(600
+I_0+U_0\big)\Delta^{-\frac{2}{5}}
\big(\tilde{I}_{k+1}^2+\tilde{U}_{k+1}^2\big)^{-1/2}\Big]
\tilde{U}_{k+1},
\end{cases}
\end{equation}
}with initial value $(I_0,U_0)$, where $\zeta_{k}$ and $\xi_{k}~(k= 1, 2, \cdot\cdot\cdot)$ represent two independent Gaussian random variables with  mean 0 and variance 1. Let the time step size $\Delta= 10^{-4}$, we carry out a detailed numerical analysis by using Matlab programming language to support theory results  and to assess the impact of the  environment noise  on Wolbachia spread in mosquito population.

\begin{expl}[Deterministic Mosquito Model]\la{exam_yhf5.1}
Let noise intensities $\sigma_I=\sigma_U=0$.  Then  stochastic  model \eqref{e2} degenerates into  deterministic model \eqref{e3}.
In  \cite{Yu2021} Hu et al. revealed that deterministic model \eqref{e3} admits three equilibria:  two locally stable equilibria $E_1(0,502)$, $E_2(400,0)$, and a saddle point $E_3(816/11,3584/11)$, see Figure \ref{fig_yhf1}. Figure \ref{fig_yhf1} plots the vector field direction of deterministic model \eqref{e3}. One observes that there exists a black separatrix  in the first quadrant. When initial value $(I_0,U_0)$ is above this separatrix, the number of infected mosquitoes declines to zero. Conversely, {when initial value $(I_0,U_0)$ is  below this separatrix, the Wolbachia spreads to the whole  mosquito population successfully. Namely, the initial infection frequency determines whether Wolbachia  invades to  wild mosquito population successfully. We refer the reader to \cite{Yu2021} for further references.}
\begin{figure}[H]
\centering
 \setlength{\abovecaptionskip}{0.cm}
\includegraphics[width=12cm,height=5.5cm]{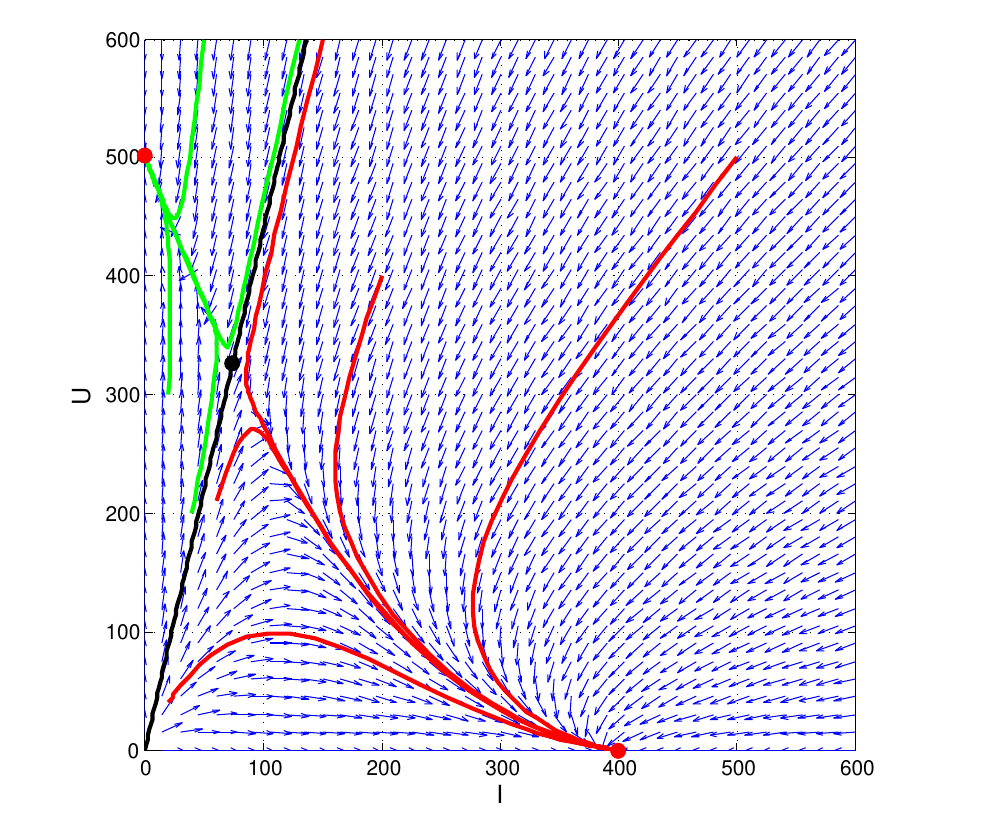}
  \caption{ 
  The vector field direction of deterministic model \eqref{e3}. The two red points $E_1$ and $E_2$ are local stable points while black point $E_3$ is a saddle point.  The black beeline passing through the black saddle point splits the first quadrant into two domains each of which contains a local stable  point. (For interpretation of the references to color in this figure legend, the reader is referred to the web version of this article.)}
  \label{fig_yhf1}
\end{figure}

Next, select the following two groups of the  initial infection frequency
  $$\mathbf{Case~1}:~\frac{I_0}{I_0+U_0}=\frac{100}{100+500};
~~~~~\mathbf{Case~2}:~\frac{I_0}{I_0+U_0}=\frac{120}{120+500}.$$
Clearly, the initial value $(100,500)$ of {\bf Case 1} is above the black separatrix. Then the
number of infected mosquitoes declines to zero, see Figure \ref{fig_yhf2}. Figure \ref{fig_yhf2}  depicts that the trajectories of $I(t)$ and $U(t)$ for {\bf Case 1}. On the other hand, the initial value $(120,500)$ of {\bf Case 2} is below the black separatrix, which implies that the Wolbachia spreads to the whole mosquito population, see Figure \ref{fig_yhf3}. Figure \ref{fig_yhf3} depicts that the  trajectories of $I(t)$ and $U(t)$ for {\bf Case 2}.

\begin{figure}[H]
  \centering
\includegraphics[width=14cm,height=4.5cm]{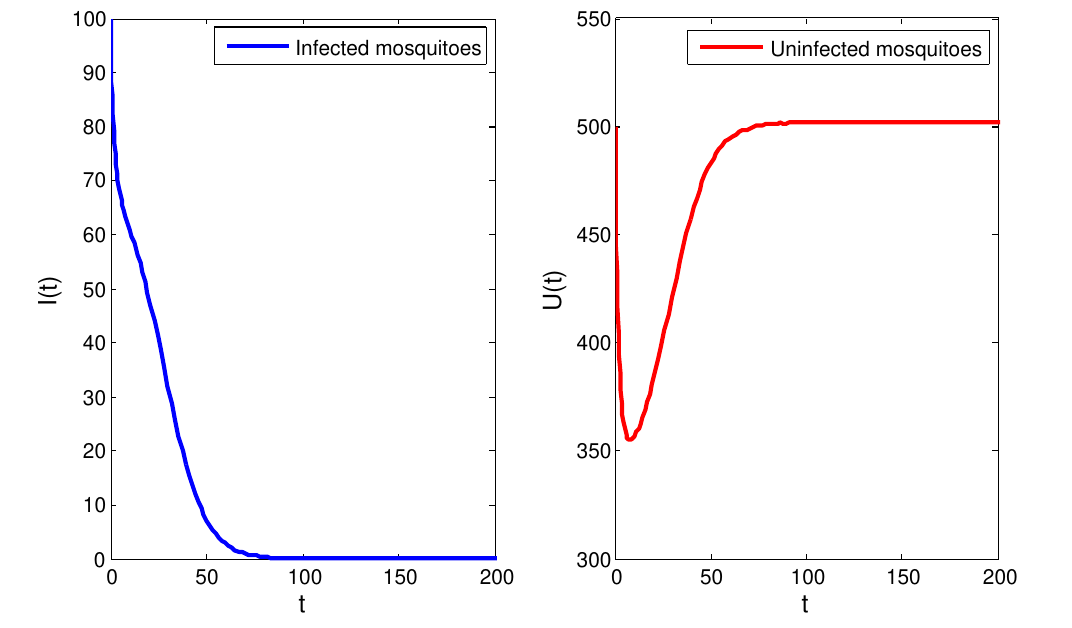}
  \caption{{\bf Case 1}. For deterministic  model \eqref{e3} with initial value $(100,500)$, the blue solid line depicts the
number  of infected mosquitoes $I(t)$; the red solid line depicts the
number  of uninfected mosquitoes $U(t)$.}
  \label{fig_yhf2}
  \vspace{-0.5em}
\end{figure}

\begin{figure}[H]
  \centering
\includegraphics[width=14cm,height=4.5cm]{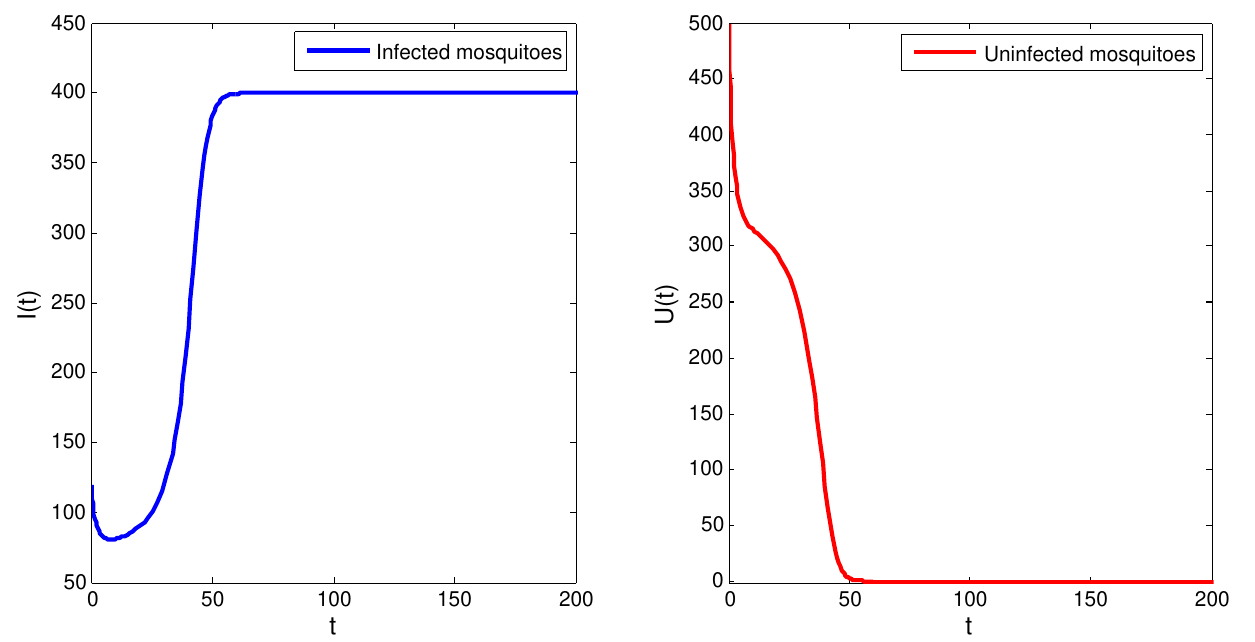}
  \caption{{\bf Case 2}. For deterministic model \eqref{e3} with initial value $(120,500)$, the blue solid line depicts the
number  of infected mosquitoes $I(t)$; the red solid line depicts the
number  of uninfected mosquitoes $U(t)$.}
  \label{fig_yhf3}
   \vspace{-1em}
\end{figure}
\end{expl}

For deterministic models \eqref{e3}, the system can be stuck in unstable points and may not converge to the stable points. However, for stochastic model \eqref{e2} derived  by adding  environment noises
to  model \eqref{e3}, the sample paths can escape from the unstable points, and concentrate
near a stable point or jump between the stable points, which leads to that model  \eqref{e2} has no curve
 analogous to the separatrix of  model \eqref{e3}. In what follows, we provide three numerical examples for stochastic model \eqref{e2} to verify our theory results.


\begin{expl}[Stochastic Mosquito Model]\la{exam_yhf5.2}

Consider model \eqref{e2} with  initial value $(100,500)$, which is  same as that in {\bf Case 1} of Example \ref{exam_yhf5.1}. 

{\bf Case 1.}  Choose   $\sigma_U=1.2$ and $\sigma_I=1$. Then
\begin{align*}
\lambda_U=-0.218<0,~~\lambda_I=-0.1<0.
\end{align*}
Theorem \ref{Th_yhf1} ${\bf (A.1)}$ tells us that
infected and uninfected mosquito populations are both extinct, see Figure \ref{fig_yhfs1}. Figure \ref{fig_yhfs1} depicts the sample paths of $I(t)$ and $U(t)$, respectively.
\begin{figure}[H]
\setlength{\abovecaptionskip}{0.cm}
\centering
\includegraphics[width=14cm,height=4.5cm]{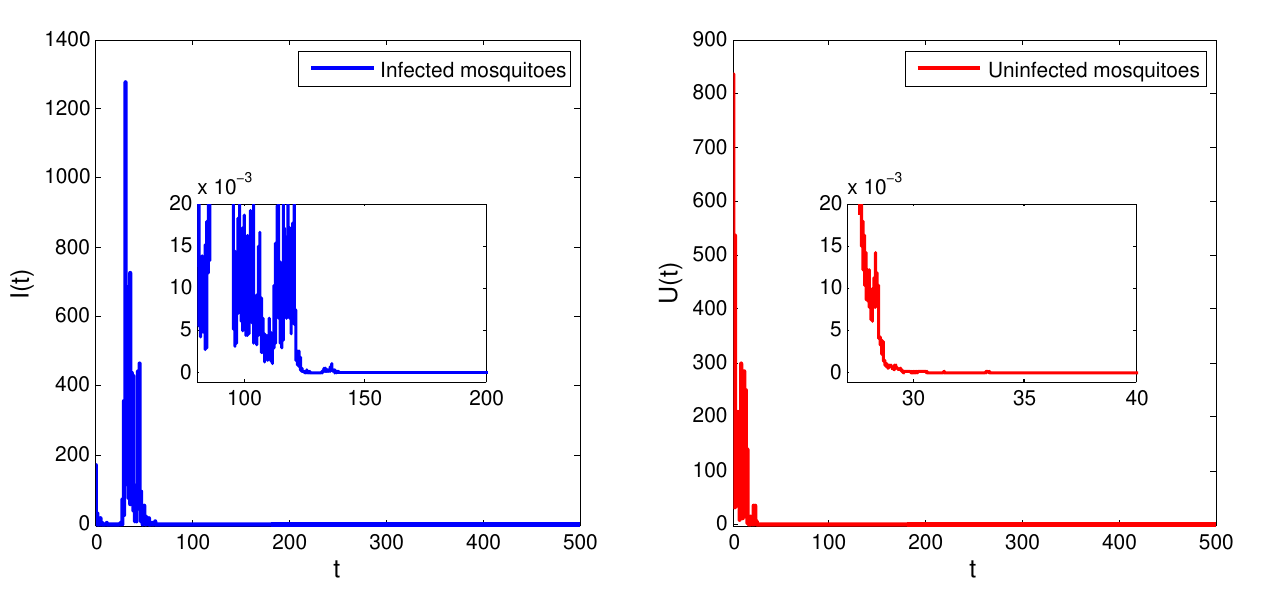}
  \caption{{\bf Case 1}.  The blue solid line depicts the
number  of infected mosquitoes $I(t)$; the red solid line depicts the
number  of uninfected mosquitoes $U(t)$.}
  \label{fig_yhfs1}
  \vspace{-0.5em}
\end{figure}

{\bf Case 2.}  Choose  $\sigma_U=1.2$ and $\sigma_I=0.2$. Then
\begin{align*}
\lambda_U=-0.218<0,~~\lambda_I=0.38>0.
\end{align*}
Theorem \ref{Th_yhf1} ${\bf (A.2)}$
reveals that  infected mosquito population is almost surely stochastically persistent and  the limit distribution of $I(t)$ is Gamma distribution $Ga(19,0.05)$, while uninfected mosquito population goes extinct exponentially fast.  Figure \ref{fig_A_2} depicts the sample paths of $I(t)$ and $U(t)$, respectively. Furthermore, using the K-S test with a significance level of $0.05$ we do confirm that the limit distribution of $I(t)$ is  $Ga(19,0.05)$  by Matlab. To make it more intuitive, we plot the density function of $Ga(19, 0.05)$ and the empirical density function of $I(t)$ in Figure \ref{fig_A_2_Fenbu}.
\begin{figure}[H]
\setlength{\abovecaptionskip}{0.cm}
  \centering
\includegraphics[width=14cm,height=4.5cm]{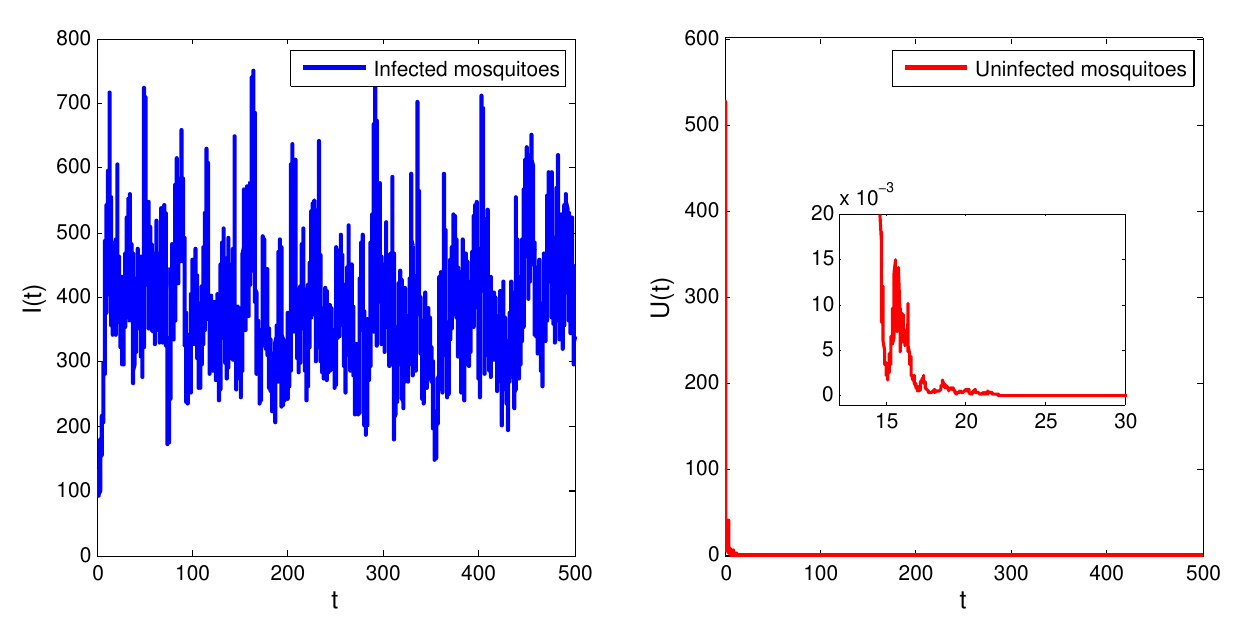}
  \caption{{\bf Case 2}. The blue solid line depicts the
number  of infected mosquitoes $I(t)$; the red solid line depicts the
number  of uninfected mosquitoes $U(t)$.}
  \label{fig_A_2}
   \vspace{-3em}
\end{figure}
\begin{figure}[H]
\setlength{\abovecaptionskip}{0.cm}
  \centering
\includegraphics[width=11cm,height=4.5cm]{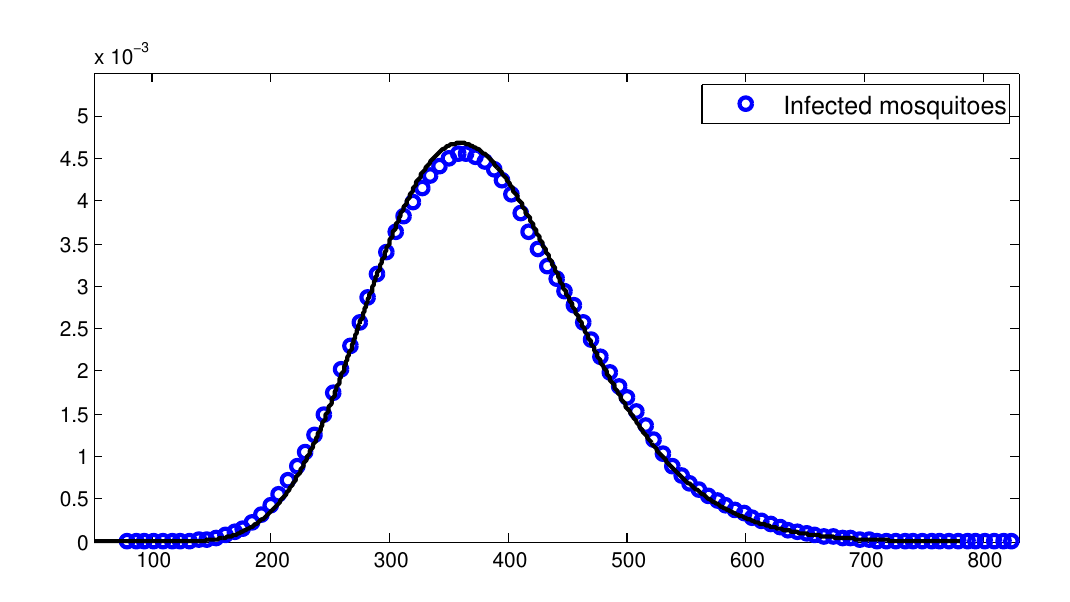}
  \caption{{\bf Case 2}. The black solid line indicates the density function of $Ga(19, 0.05)$, the blue dotted line indicates the empirical density function of  $I(t)$.}
  \label{fig_A_2_Fenbu}\vspace{-0.9em}
\end{figure}

\begin{rem}
Under the same initial value as that in {\bf Case 1} of Example \ref{exam_yhf5.1}, Example \ref{exam_yhf5.2} exhibits different completely  dynamical behaviors of stochastic model \eqref{e2} compared with those of deterministic model \eqref{e3}. In particular, for a low initial infection frequency, proper noise intensities $\sigma_I$ and $\sigma_U$ still drive a successful Wolbachia invasion into mosquito population. This implies that  environment noises can not be ignored.
\end{rem}

\end{expl}

Next, for another initial  value $(120,500)$, we go a further step to compare the dynamical behaviors of stochastic model \eqref{e2} and  deterministic  model \eqref{e3}.
\begin{expl}[Stochastic Mosquito Model]\la{exam_yhf5.3}
Consider model \eqref{e2} with  initial value $(120,500)$, which is same as that in {\bf Case 2} of Example \ref{exam_yhf5.1}.  


{\bf Case 1.}  Choose  $\sigma_U=0.5$ and $\sigma_I=0.1$. Then
\begin{align*}  \lambda_U=0.377>0,~~\lambda_I/d_I=395>377=\lambda_U/d_U.
\end{align*}
Theorem \ref{Th_yhf1} ${\bf (B.1)}$
reveals that  infected mosquito population is almost surely stochastically persistent and  the limit distribution of $I(t)$ is  $Ga(79,0.2)$, while uninfected mosquito population goes extinct exponentially fast.  Figure \ref{fig_B_1} depicts the sample paths of $I(t)$ and $U(t)$, respectively. Similarly, by Matlab we verify that the limit distribution of $I(t)$ is $Ga(79,0.2)$. In addition, we plot the density function of $Ga(79, 0.2)$ and the empirical density function of $I(t)$ in Figure \ref{fig_B_1_Fenbu}.
\begin{figure}[H]
\setlength{\abovecaptionskip}{0.cm}
  \centering
\includegraphics[width=14cm,height=4.5cm]{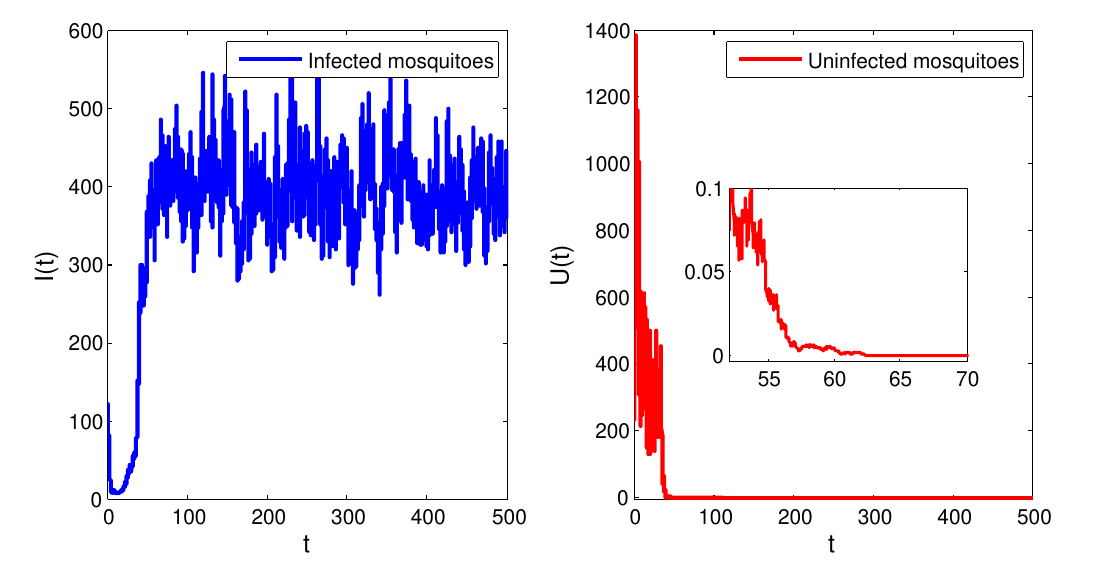}
  \caption{{\bf Case 1}. The blue solid line depicts the
number  of infected mosquitoes $I(t)$; the red solid line depicts the
number  of uninfected mosquitoes $U(t)$.}
  \label{fig_B_1}
   \vspace{-2em}
\end{figure}
\begin{figure}[H]
\setlength{\abovecaptionskip}{0.cm}
  \centering
\includegraphics[width=10cm,height=5cm]{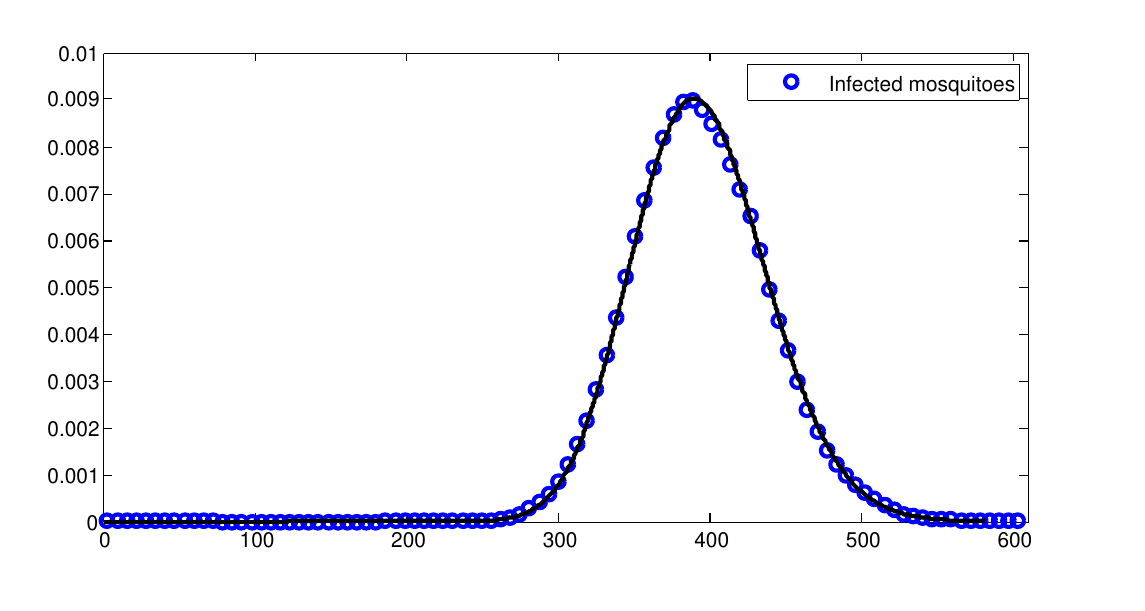}
  \caption{{\bf Case 1}. The black solid line indicates the density function of $Ga(79, 0.2)$, the blue dotted line indicates the empirical density function of  $I(t)$.}
  \label{fig_B_1_Fenbu}
\end{figure}

{\bf Case 2.} Choose   $\sigma_U=0.5$ and $\sigma_I=1.1$. Then
\begin{align*}
\lambda_U=0.377>0,~~\lambda_I=-0.205<-0.173=\lambda_U-b_U.
  \end{align*}
Theorem \ref{Th_yhf1} ${\bf (B.2)}$
shows that infected mosquito population
 goes  extinct exponentially fast and uninfected mosquito population is almost surely stochastically persistent, see Figure \ref{fig_B_2}. Figure \ref{fig_B_2} depicts the sample paths of $I(t)$ and $U(t)$, respectively.
\begin{figure}[H]
  \centering
\includegraphics[width=15cm,height=4.5cm]{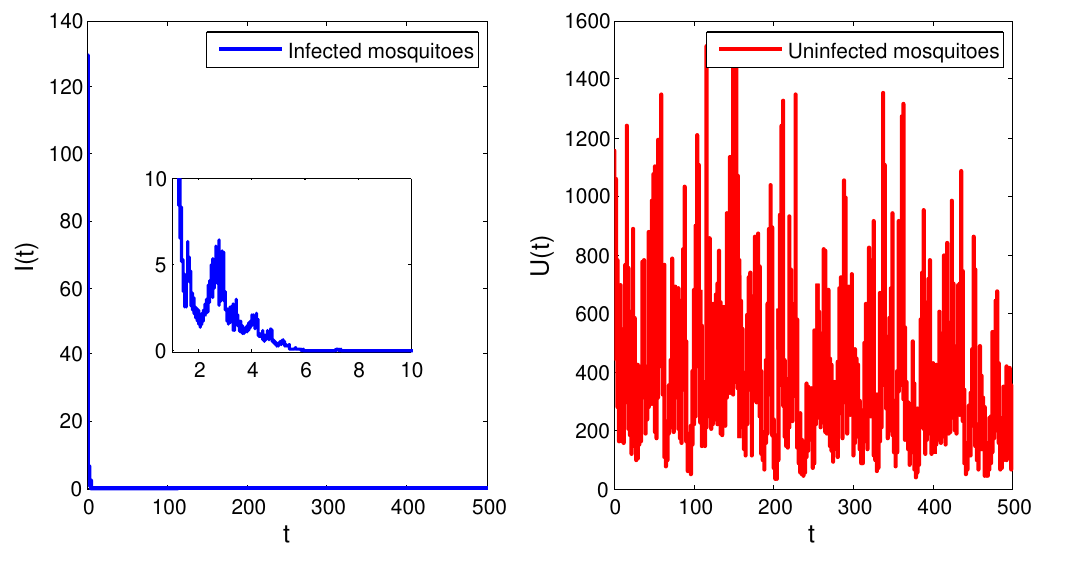}
  \caption{{\bf Case 2}.  The blue solid line depicts the
number of infected mosquitoes $I(t)$; the red solid line depicts the
number of uninfected mosquitoes $U(t)$.}
  \label{fig_B_2}
\end{figure}

{\bf Case 3.} Choose  $\sigma_U=0.5$ and $\sigma_I=0.6$. Then
\begin{align*}
\lambda_U=0.377>0,~~~\frac{\lambda_I}{d_I}=220\leq\frac{\lambda_U}{d_U}=377.
\end{align*}
Theorem \ref{Th_yhf1} ${\bf (B.3)}$  shows that any stationary distribution of  $(I(t), U(t))$ has no support on $\mathbb{R}^{2,\circ}_{+}$, see Figure \ref{fig_yhfs10}. Figure \ref{fig_yhfs10} depicts the empirical density function of $(I(t), U(t))$. Meanwhile, Theorem \ref{Th_yhf1} ${\bf (B.3)}$ implies that infected and uninfected mosquito popilations are impossible to coexist in the long term. To more intuitively exhibit this result,  we further plot  multiple sample paths of $I(t)$ and $U(t)$ in Figure \ref{fig_yhfs8}.
\begin{figure}[H]
  \centering
\includegraphics[width=8cm,height=4.5cm]{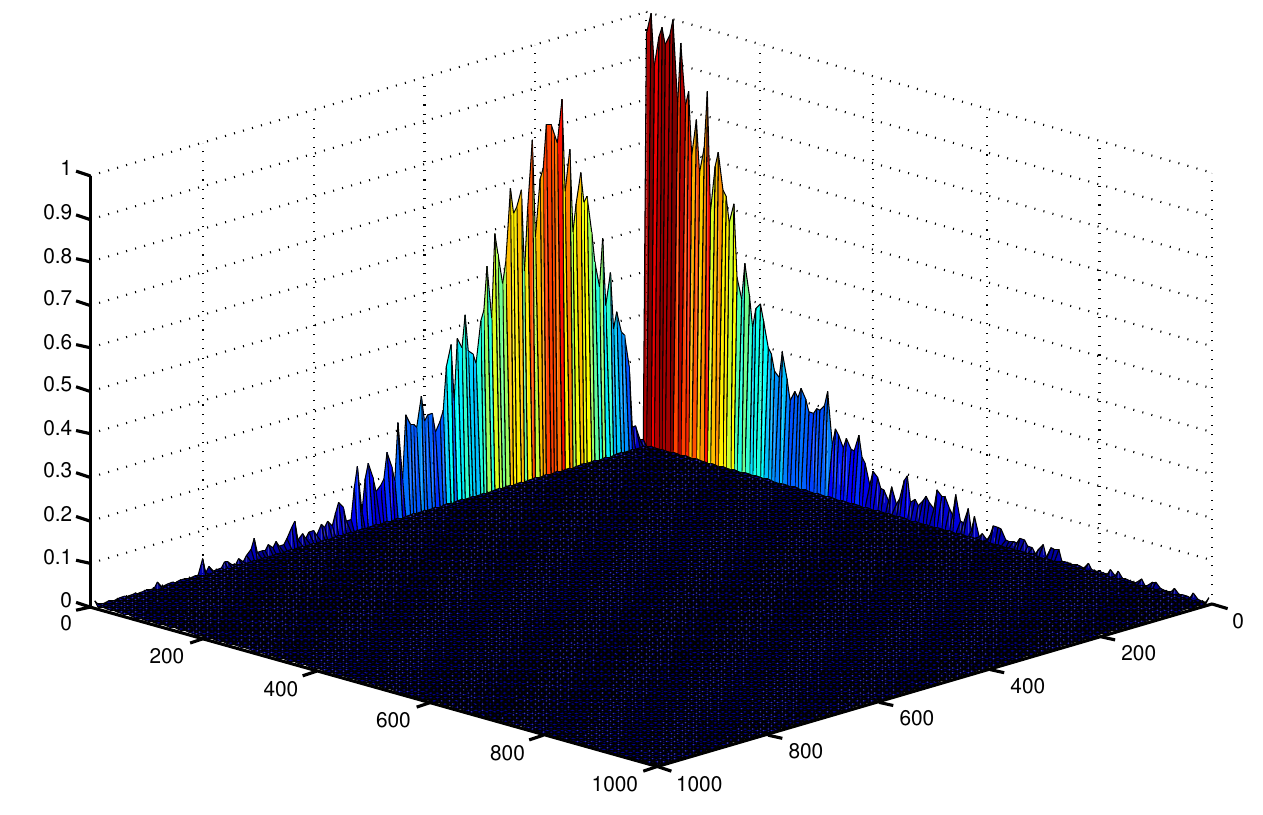}
  \caption{{\bf Case 3}. The empirical density function of  $(I(t),U(t))$. }
  \label{fig_yhfs10}
\end{figure}
\begin{figure}[H]
   \vspace{-0.5em}
  \centering
\includegraphics[width=15cm,height=6cm]{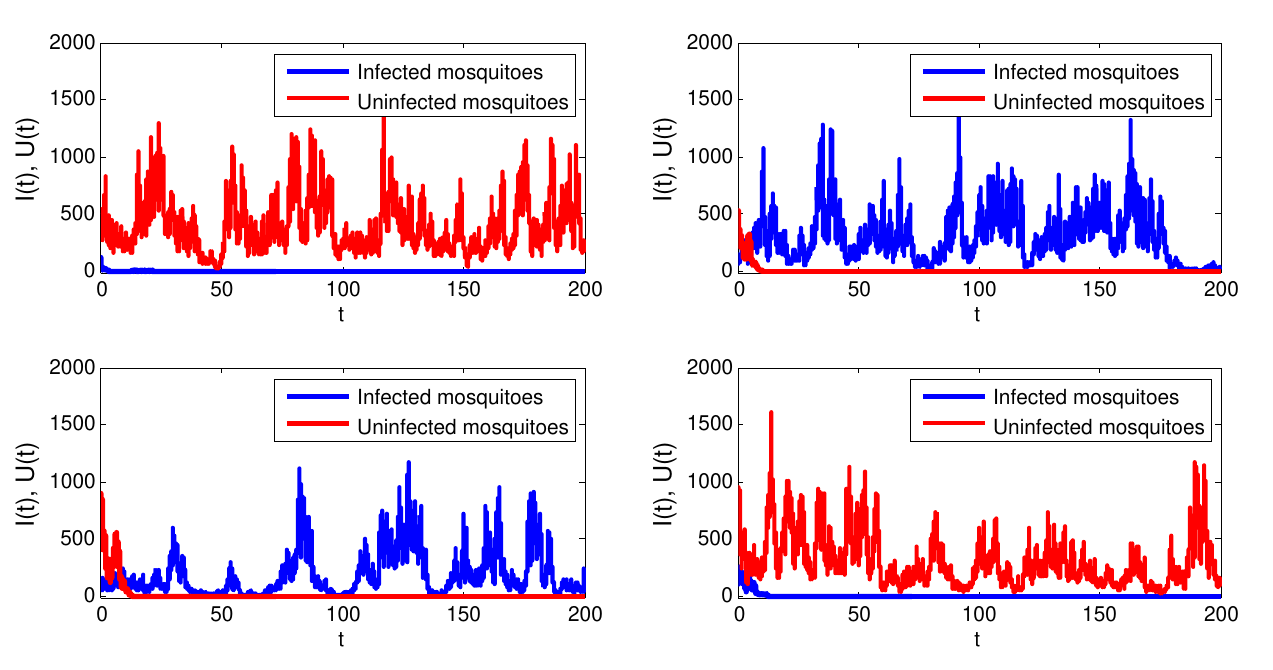}
  \caption{{\bf Case 3}. The blue solid line depicts the
number  of infected mosquitoes $I(t)$; the red solid line depicts the
number  of uninfected mosquitoes $U(t)$.}
  \label{fig_yhfs8}
\end{figure}

In addition, Theorem \ref{L3.6} gives that $\mathbb{E}(I(t)\wedge U(t))$  is declining to 0, which implies that at least one kind of mosquito population is extinct, see Figure \ref{fig_yhfs9}. Figure \ref{fig_yhfs9} depicts 4 sample paths of $I(t)\wedge U(t)$ and  the trajectory of $\mathbb{E}(I(t)\wedge U(t))$, respectively.

\begin{figure}[H]
\hspace{6ex}\raggedright
\begin{minipage}[t]{0.4\linewidth}
\setlength{\abovecaptionskip}{0pt}
\setlength{\belowcaptionskip}{10pt}
    \centering
    \includegraphics[width=7cm,height=4.5cm]{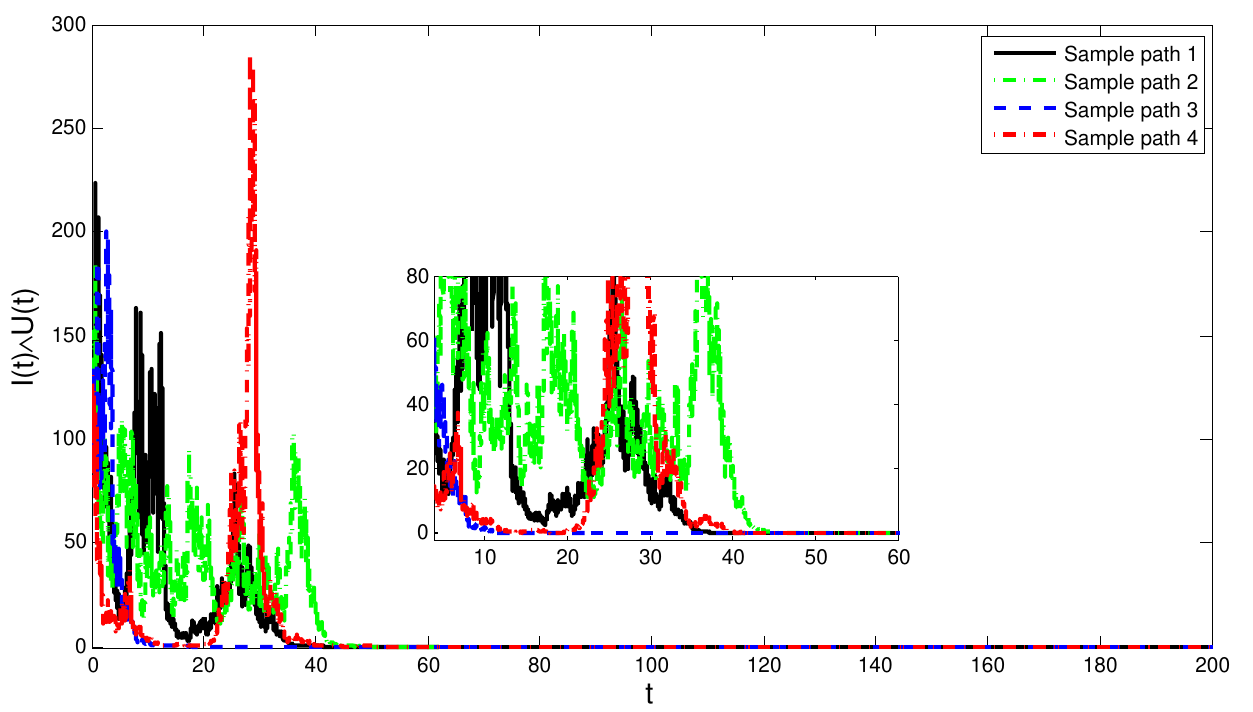}
\end{minipage}
\hspace{5ex}   
\begin{minipage}[t]{0.4\linewidth}
\setlength{\abovecaptionskip}{0pt}
\setlength{\belowcaptionskip}{10pt}
   \raggedleft
    \includegraphics[width=7cm,height=4.5cm]{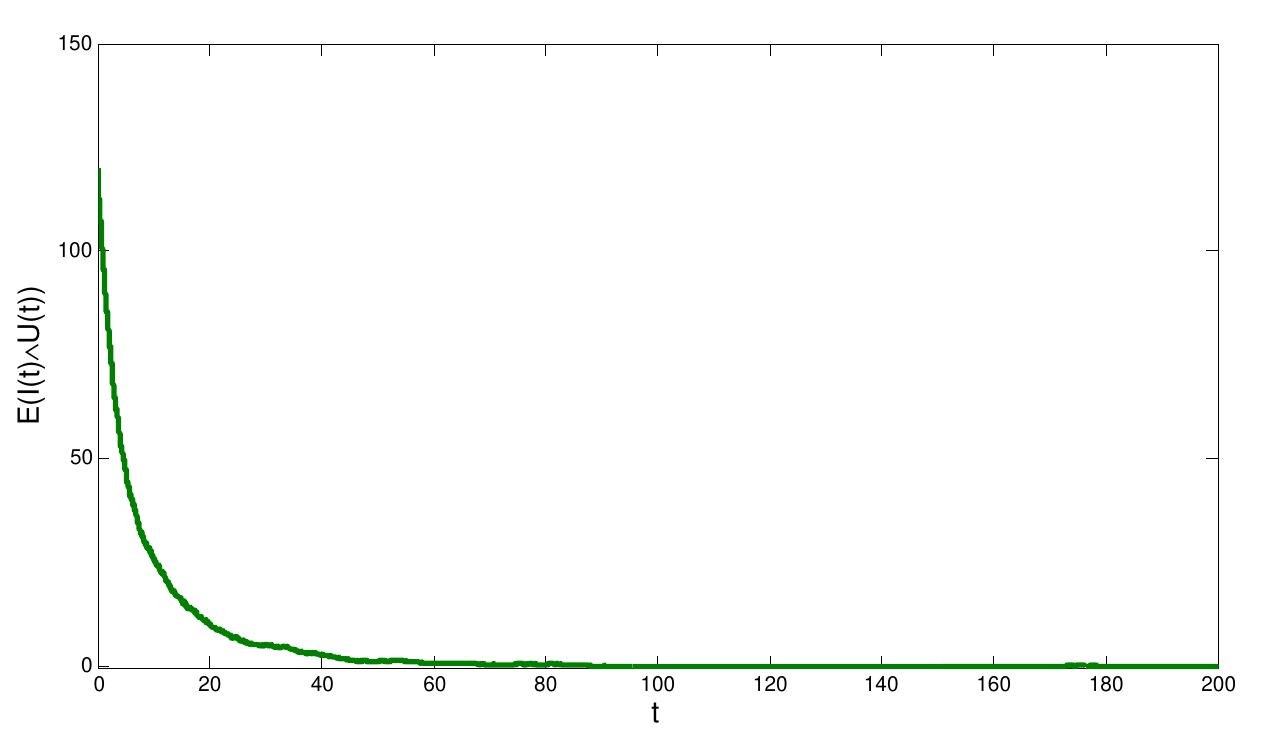}
\end{minipage}
\setlength{\abovecaptionskip}{0pt}
 \caption{{\bf Case 3}. Left: 4 sample paths of $I(t)\wedge U(t)$. Right: The trajectory of $\mathbb{E}(I(t)\wedge U(t))$.}\label{fig_yhfs9}
\end{figure}
\begin{rem}
Under the same initial value as that in {\bf Case 2} of Example \ref{exam_yhf5.1},
 {\bf Case 2} and {\bf Case 3} of Example \ref{exam_yhf5.3} also present  different completely dynamical behaviors. It is worth noting that although for a high initial infection frequency,  infected mosquito population will still be extinct if  noise intensity $\sigma_I$ is large sufficiently, see {\bf Case 2} of Example \ref{exam_yhf5.3}.
\end{rem}
\end{expl}

Quantitative dynamical features of model \eqref{e2}  are described by
  Theorem \ref{Th_yhf1}, whereas  we are currently unable to precisely determine the weights of three boundary measures theoretically in Theorem \ref{Th_yhf1} {\bf (B.3)}. To understand this intuitively, we simulate numerically the limit distributions of $(I(t),U(t))$  for different initial values $(I_0,U_0)$.
\begin{expl}[Stochastic Mosquito Model]\la{exam_yhf5.4}
Keep the noise intensities $\sigma_U$ and $\sigma_{I}$ are same as those in {\bf Case 3} of Example \ref {exam_yhf5.3}. We  plot the empirical density functions of $(I(t),U(t))$ for different initial values $(I_0,U_0)$ in Figures \ref{fig_B_3_Fenbu_13} and \ref{fig_B_3_Fenbu_14}. In addition, Figure \ref{fig_yhfs8} depicts the empirical density function of $(I(t),U(t))$ with initial value $(120,500)$ while Figure \ref{fig_B_3_Fenbu_14} (Right) depicts the empirical density function of $(I(t),U(t))$ with initial value $(12,50)$. It is evident to see that these two density pictures are very similar. One observes from Figure \ref{fig_yhfs8}, Figure \ref{fig_B_3_Fenbu_13}  and Figure \ref{fig_B_3_Fenbu_14} that the weights of boundary measures described in  Theorem \ref{Th_yhf1} {\bf (B.3)} will vary with the initial infection frequency.
\begin{figure}[H]
\hspace{6ex}\raggedright
\begin{minipage}[t]{0.4\linewidth}
\setlength{\abovecaptionskip}{0pt}
\setlength{\belowcaptionskip}{10pt}
    \centering
    \includegraphics[width=7cm,height=4.5cm]{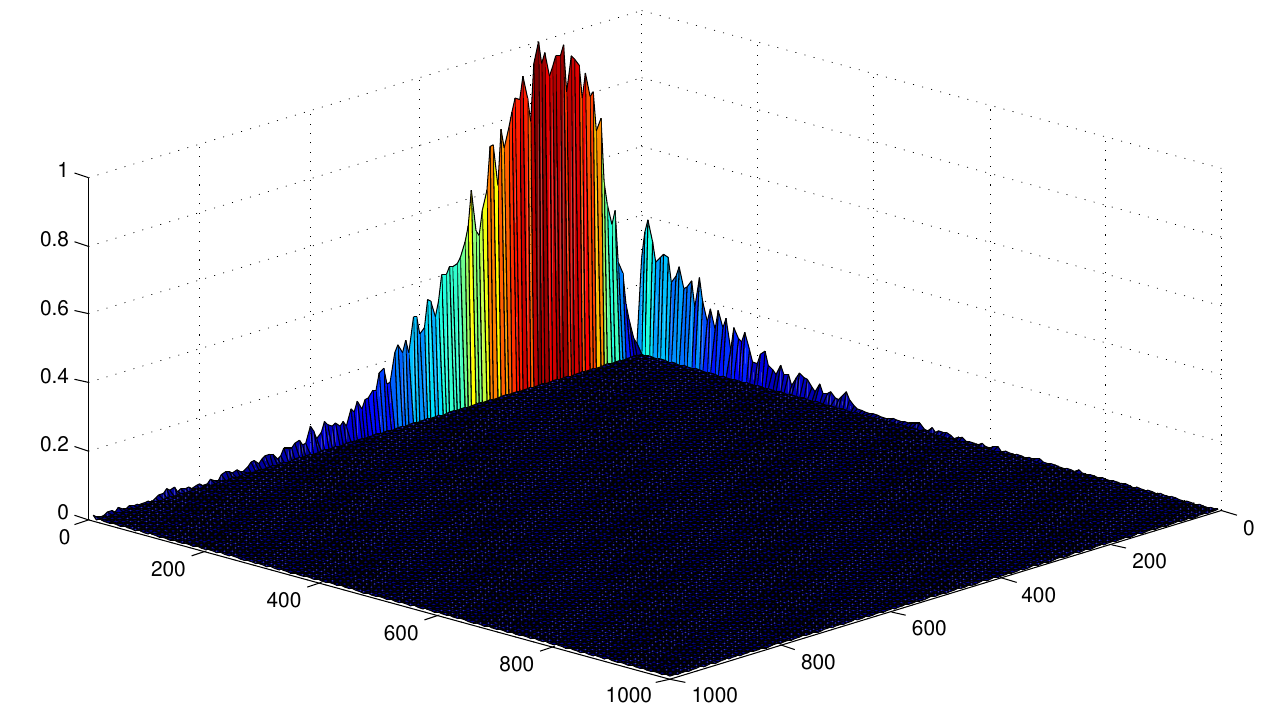}
\end{minipage}
\hspace{5ex}   
\begin{minipage}[t]{0.4\linewidth}
\setlength{\abovecaptionskip}{0pt}
\setlength{\belowcaptionskip}{10pt}
   \raggedleft
    \includegraphics[width=7cm,height=4.5cm]{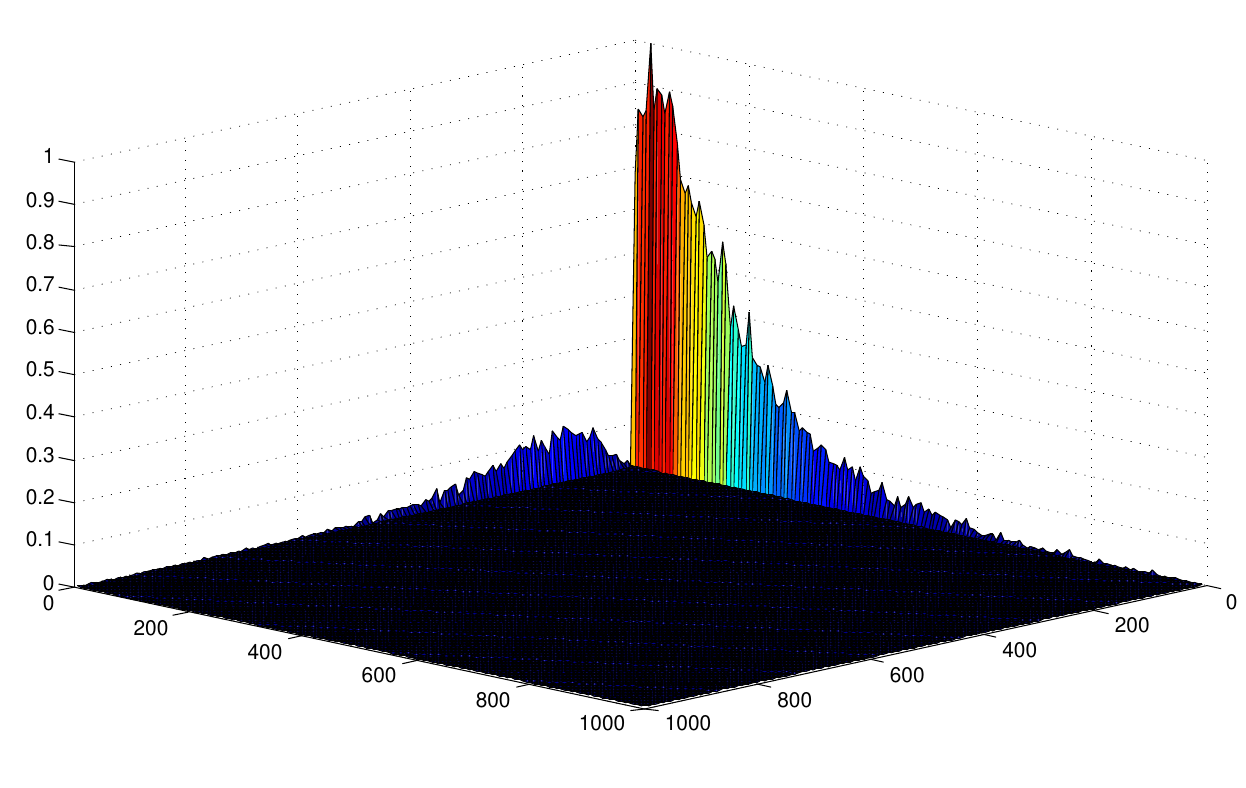}
\end{minipage}
\setlength{\abovecaptionskip}{0pt}
 \caption{{\bf Case 3}. Left: The empirical density function of  $(I(t),U(t))$ for initial value $(10,500)$. Right: The empirical density function of  $(I(t),U(t))$ for initial value $(100,50)$.}
 \label{fig_B_3_Fenbu_13}
\end{figure}

\begin{figure}[H]
\hspace{6ex}\raggedright
\begin{minipage}[t]{0.4\linewidth}
\setlength{\abovecaptionskip}{0pt}
\setlength{\belowcaptionskip}{10pt}
    \centering
    \includegraphics[width=7cm,height=4.5cm]{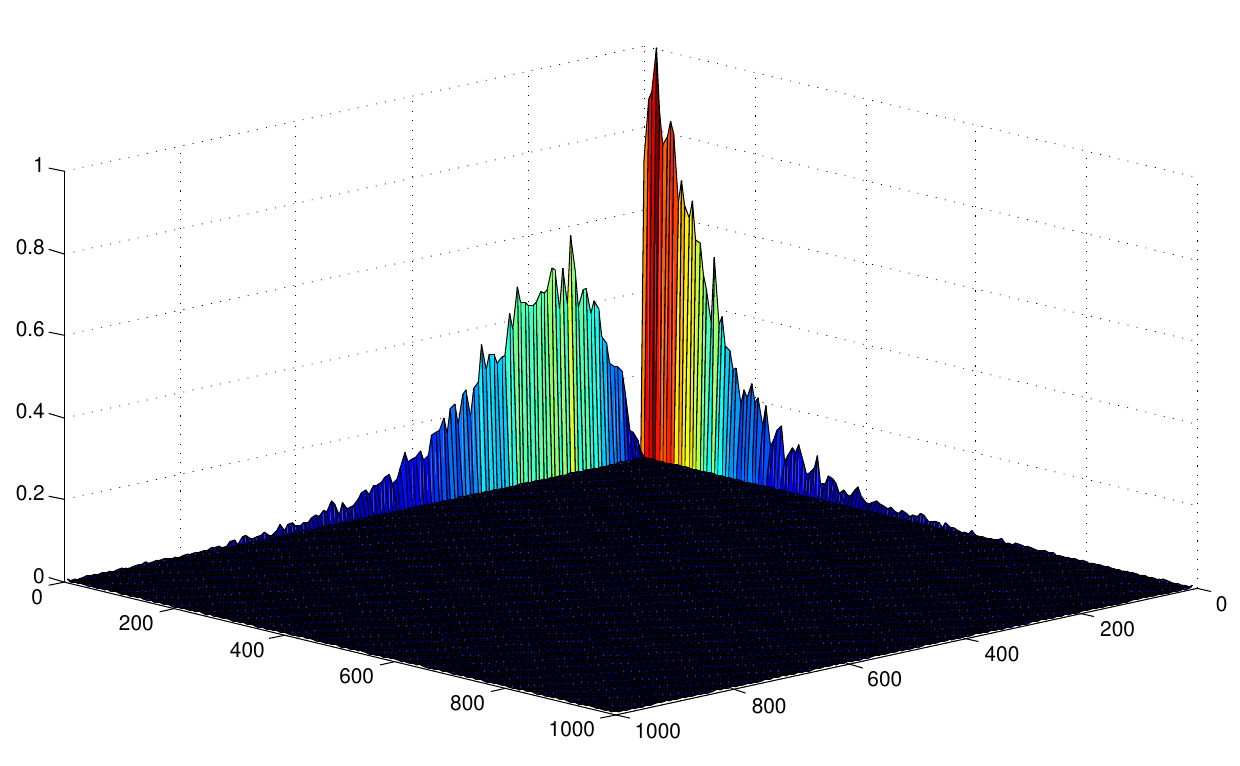}
\end{minipage}
\hspace{5ex}   
\begin{minipage}[t]{0.4\linewidth}
\setlength{\abovecaptionskip}{0pt}
\setlength{\belowcaptionskip}{10pt}
   \raggedleft
    \includegraphics[width=7cm,height=4.5cm]{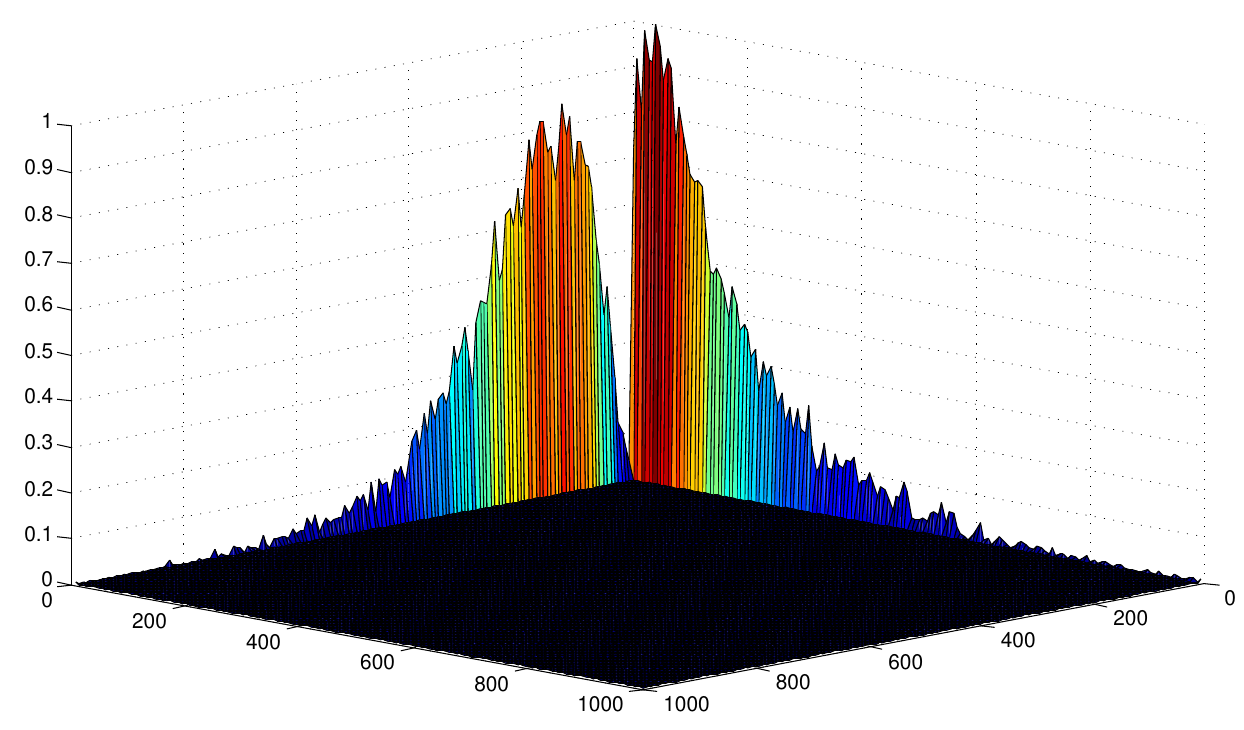}
\end{minipage}
\setlength{\abovecaptionskip}{0pt}
 \caption{{\bf Case 3}. Left: The empirical density function of  $(I(t),U(t))$ for initial value $(100,500)$. Right: The empirical density function of  $(I(t),U(t))$ for initial value $(12,50)$.}
 \label{fig_B_3_Fenbu_14}
\end{figure}
\end{expl}

\section*{Acknowledgements}

{The authors would like to thank Professor Jifa Jiang for helpful discussions and valuable comments during the preparation of this
manuscript.

Research of Xiaoyue Li was supported by the National Natural
Science Foundation of China (No. 11971096), the National Key
R\&D Program of China (2020YFA0714102), the Natural Science
Foundation of Jilin Province, China (No. YDZJ202101ZYTS154),  and the Fundamental Research Funds for the
Central Universities, China.

Xuerong Mao would like to thank the Royal Society
(WM160014, Royal Society Wolfson Research Merit Award), the
Royal Society and the Newton Fund (NA160317, Royal Society-Newton Advanced Fellowship), the Royal Society of Edinburgh (RSE1832), and Shanghai Administration of Foreign Experts Affairs (21WZ2503700, the Foreign Expert Program) for their financial support.

Hongfu Yang would like to thank the National Natural Science
Foundation of China (No. 12101144), and
the Natural Science Foundation of Guangxi Province (No. 2021GXNSFBA196080) for their financial support.
}




\end{document}